\documentclass[reqno,11pt]{amsart}
\usepackage{amsthm,amsfonts,amssymb,euscript,
mathrsfs,graphics,color,amsmath,latexsym,marginnote}
\usepackage{dsfont}   
 
\theoremstyle{plain}

\usepackage{hyperref}
\usepackage{times}
\usepackage{mathtools}

\numberwithin{equation}{section}

\newtheorem{theorem}{Theorem}[section]
\newtheorem{proposition}[theorem]{Proposition}
\newtheorem{lemma}[theorem]{Lemma}

\newtheorem{remark}[theorem]{Remark}
\newtheorem{remarks}[theorem]{Remark}
\newtheorem{hypo}[theorem]{Hypothesis}
\newtheorem{definition}[theorem]{Definition}
\newcommand{\be}{\begin{equation}}
\newcommand{\ee}{\end{equation}}

\newcommand{\uno}{\mathds{1}}
\newcommand{\e}{\varepsilon}

\newcommand{\ov}{\overline}

\newcommand{\R}{\mathbb R}
\newcommand{\C}{\mathbb C}

\newcommand{\Z}{\mathbb Z}

\newcommand{\N}{\mathbb N}
\newcommand{\T}{\mathbb T}

\newcommand{\s }{\sigma }
\newcommand{\ii }{{\rm i} }

\newcommand{\x }{\xi }

\newcommand{\pa}{\partial}

\newcommand{\opbw}{{Op^{\mathrm{BW}}}}

\def\hat{\widehat}

\def\bar{\overline}

\def\ba{\begin{aligned}}
\def\ea{\end{aligned}}
\def\beginm{\begin{multline}}
\def\endm{\end{multline}}

\providecommand{\vect}[2]{{\bigl[\begin{smallmatrix}#1\\#2\end{smallmatrix}\bigr]}}   
\providecommand{\sm}[4]{{\bigl[\begin{smallmatrix}#1&#2\\#3&#4\end{smallmatrix}\bigr]}}

\makeatletter

\def\l@subsection{\@tocline{2}{0pt}{2.5pc}{5pc}{}}
\def\l@subsubsection{\@tocline{3}{0pt}{4.5pc}{5pc}{}}
\renewcommand\tocchapter[3]{%
  \indentlabel{\@ifnotempty{#2}{\ignorespaces#2.\quad}}#3%
}
\newcommand\@dotsep{4.5}
\def\@tocline#1#2#3#4#5#6#7{\relax
  \ifnum #1>\c@tocdepth 
  \else
    \par \addpenalty\@secpenalty\addvspace{#2}%
    \begingroup \hyphenpenalty\@M
    \@ifempty{#4}{%
      \@tempdima\csname r@tocindent\number#1\endcsname\relax
    }{%
      \@tempdima#4\relax
    }%
    \parindent\z@ \leftskip#3\relax \advance\leftskip\@tempdima\relax
    \rightskip\@pnumwidth plus1em \parfillskip-\@pnumwidth
    #5\leavevmode\hskip-\@tempdima{#6}\nobreak
    \leaders\hbox{$\m@th\mkern \@dotsep mu\hbox{.}\mkern \@dotsep mu$}\hfill
    \nobreak
    \hbox to\@pnumwidth{\@tocpagenum{#7}}\par
    \nobreak
    \endgroup
  \fi}
\def\l@subsection{\@tocline{2}{0pt}{2.5pc}{5pc}{}}

\begin{document}
\bibliographystyle{plain}

\title[Well-posedness for quasilinear NLS]{Local well-posedness for the quasi-linear Hamiltonian Schr\"odinger equation on tori}

\date{}

\author{Roberto Feola}
\address{\scriptsize{Dipartimento di Matematica, Universit\`a degli Studi di Milano, Via Saldini 50, I-20133}}
\email{roberto.feola@unimi.it}


\author{Felice Iandoli}
\address{\scriptsize{Laboratoire Jacques Louis Lions (Sorbonne Universit\'e), 5 Place Jussieu, Paris 75005.  }}
\email{felice.iandoli@sorbonne-universite.fr}



\thanks{Felice Iandoli has been supported  by ERC grant ANADEL 757996.
Roberto Feola has been supported by 
 the Centre Henri Lebesgue ANR-11-LABX- 0020-01 
 and by ANR-15-CE40-0001-02 ``BEKAM'' of the ANR}

\begin{abstract}  
We prove a local in time well-posedness result for quasi-linear Hamiltonian Schr\"odinger equations on $\T^d$ for any $d\geq 1$. For any initial condition in the Sobolev space $H^s$, with $s$ large, we prove the existence and uniqueness of classical solutions of the Cauchy problem associated to the equation. The lifespan of such a solution depends only on the size of the initial datum.  Moreover we prove the continuity of the solution map. 
\end{abstract}  


\maketitle
\tableofcontents

\section{Introduction}
In this paper we  study the local in time solvability of the Cauchy problem associated to the following quasi-linear  perturbation of the Schr\"odinger equation
\begin{equation}\label{QLNLS}
\ii u_{t}-\Delta u+P(u)=0 \,,\qquad 
u=u(t,x)\,, 
\qquad x=(x_1,\ldots,x_{d})\in \mathbb{T}^{d}:=(\mathbb{R}/2\pi \mathbb{Z})^{d}
 \end{equation}
 with
   \begin{equation}\label{NLSnon}
   \begin{aligned}
P(u)
&:=(\pa_{\bar{u}}F)(u,\nabla u)-\sum_{j=1}^{d}\pa_{x_j}\big(\pa_{\bar{u}_{x_{j}}}F\big)(u,\nabla u)\,,
\end{aligned}
 \end{equation}
where we denoted $\pa_{u}:=(\pa_{{\rm Re}(u)}-\ii \pa_{{\rm Im}(u)})/2$ and  
$\pa_{\bar{u}}:=(\pa_{{\rm Re}(u)}+\ii \pa_{{\rm Im}(u)})/2$
the Wirtinger derivatives. The function
$F(y_0,y_1,\ldots,y_{d})$ is in $ C^{\infty}(\mathbb{C}^{d+1},\mathbb{R})$
   in the \emph{real} sense, i.e. $F$ is $C^{\infty}$
  as function of ${\rm Re}(y_i)$, ${\rm Im}(y_i)$. Moreover 
  we assume that
  $F$ has a zero of order at least $3$ at the origin, so that 
   $P$ has a zero of order at least $2$ at the origin.
 Here $
 \nabla u=(\pa_{x_{1}} u, \ldots, \pa_{x_d}u)
 $ is the gradient 
 and $\Delta$ denotes the Laplacian operator defined by linearity as
 \[
 \Delta e^{\ii j \cdot x}=-|j|^{2}e^{\ii j\cdot x}\,, \quad \forall\, j\in \mathbb{Z}^{d}\,.
 \]
 Notice that equation \eqref{QLNLS} is \emph{Hamiltonian}, i.e.
 \begin{equation}\label{HamNLS}
 u_{t}=\ii \nabla_{\bar{u}}H(u,\bar{u})\,,
\quad 
H(u,\bar{u}):=\int_{\mathbb{T}^{d}}|\nabla u|^{2}+F(u,\nabla u) dx\,,
\end{equation}
where $\nabla_{\bar{u}}:=(\nabla_{{\rm Re}(u)}-\ii\nabla_{{\rm Im}(u)})/2$
and $\nabla_{{\rm Re}(u)}$, $\nabla_{{\rm Im}(u)}$ denote the $L^{2}$-gradient. 
In order to be able to consider initial data with big size we assume  that the function $F$, defining the nonlinearity, satisfies 
following \emph{ellipticity condition}.
\begin{hypo}{\bf (Global ellipticity).}\label{hyp1}
We assume that there exist constants $\mathtt{c}_1,\mathtt{c}_2>0$
such that the following holds.
For any $\x=(\x_1,\ldots,\x_{d})\in \mathbb{R}^{d}$, 
$y=(y_0,\ldots,y_d)\in \mathbb{C}^{d+1}$
one has
\begin{equation}\label{ellipthypo1}
\sum_{j,k=1}^{d}\x_{j}\x_{k}\Big(\delta_{jk}+\pa_{y_j}\pa_{\ov{y_{k}}}F(y)\Big) 
\geq \mathtt{c}_1|\x|^{2}\,,
\end{equation}
\begin{equation}\label{ellipthypo2}
\Big(1+|\x|^{-2}\sum_{j,k=1}^{d}\x_{j}\x_{k}\pa_{y_j}
\pa_{\ov{y_{k}}}F(y)\Big)^{2}
-\Big| |\x|^{-2}\sum_{j,k=1}^{d}\x_{j}\x_{k}\pa_{\ov{y_j}}\pa_{\ov{y_{k}}}F(y) 
\Big|^{2}
\geq \mathtt{c}_2\,,
\end{equation}
where $\delta_{jj}=1$, $\delta_{jk}=0$ for $j\neq k$.
\end{hypo}
The main result of this paper is the following.
\begin{theorem}{\bf (Local well-posedness).}\label{main}
Let $F$ be a function satisfying the Hypothesis \ref{hyp1}. 
For any $s> 2d+11$ the following holds true.
Consider the equation \eqref{QLNLS}
with initial condition $u(0,x)=u_0(x)$ in $H^s(\mathbb{T}^d;\mathbb{C})$, 
then there exists a time 
$0<T=T(\|u_0\|_{H^{s}})$ 
and a unique solution 
\[
u(t,x)\in C^0([0,T),H^s(\mathbb{T}^d;\mathbb{C}))
\cap C^1([0,T),H^{s-2}(\mathbb{T}^d;\mathbb{C})).
\]
Moreover the solution map $u_0(x)\mapsto u(t,x)$ is continuous with respect to the $H^s$ topology for any $t$ in $[0,T)$.
\end{theorem}
In the following  we make some comments about the result we obtained.
\begin{itemize}
\item In the case of small initial conditions, i.e. $\|u_0\|_{H^s}\ll 1$, 
one can disregard the global ellipticity Hypothesis \ref{hyp1}. 
Indeed for $``u$ small" the nonlinearity $F$ is always 
\emph{locally elliptic} and one can prove the theorem in a similar way.

\item When the initial condition satisfies $\|u_0\|_{H^s}\sim\e\ll 1$
it turns out that the life-span, implicit in this work, is $T\sim \e^{-1}$.
We do not know if the solutions are globally in time defined or not.
There are positive results when equation \eqref{QLNLS} is posed on $\mathbb{R}^{d}$.
We refer, for instance, to  \cite{saut-glob} by De Bouard-Hayashi-Saut and 
 \cite{saut-scatt} by De Bouard-Hayashi-Naumkin-Saut.
 However in the latter papers the results are based on dispersive estimates
 which are not available on compact manifolds.
 Our result sets the stage 
for a refined normal form analysis to improve the time of existence,
 assuming certain non-resonance conditions.
This is the content of the recent paper \cite{FGI20}.

\item Theorem \ref{main} provides the well-posedness 
for the most general quasi-linear,  Hamiltonian  (local) nonlinearity.
An example of physically relevant 
(see  \cite{sergio,fedayn, goldman, hasse})
 quasi-linear Schr\"odinger is 
\[
\ii u_t-\Delta u
+\big[\Delta (h(|u|^2))\big]h'(|u|^2)u-|u|^2u=0\,,
\]
for some real-valued $C^{\infty}$ function $h$. 
This model has been studied by many mathematicians,   
we refer for instance to \cite{colinjj} by Colin-Jeanjean 
and reference therein.

\item We did not attempt to achieve the 
theorem in the best possible regularity $s$. 
We work in high regularity in order  to perform 
suitable changes of coordinates and having a 
symbolic calculus at a sufficient order, which 
requires smoothness of the functions of the phase space.
\item We prove the continuity of the solution map, 
we do not know if it is uniformly continuous or not. 
Unlike the semi-linear case (for which we refer to \cite{caze}), 
it is an hard problem to establish if the flow is more regular. 
These problems have been discussed in the paper \cite{MST} 
about Benjamin-Ono and related equations by 
Molinet-Saut-Tzvetkov. We also quote the 
survey article \cite{tzvetkov} by Tzvetkov.
\end{itemize}
To the best of our knowledge this theorem 
is the first of this kind on a compact manifold 
of dimension greater than or equal to $2$. 
For the same equation on the circle we 
quote our paper \cite{Feola-Iandoli-Loc} and 
the one by Baldi-Haus-Montalto \cite{BHM}. 
In \cite{BHM} a Nash-Moser iterative scheme has 
been used in order to obtain the existence of 
solutions in the case of \emph{small} initial conditions. 
In our previous paper  \cite{Feola-Iandoli-Loc} we exploited 
the fact that in dimension one it is possible to 
conjugate the equation to  constant coefficients 
by means of para-differential changes of coordinates. 
This technique has been used in several other papers 
to study the normal forms associated to these quasi-linear 
equations we quote for instance our papers 
\cite{Feola-Iandoli-Long, Feola-Iandoli-Totale}, 
and the earlier one by Berti-Delort \cite{BD} on the 
gravity-capillary water waves system. 
The proof we provide here is not based on 
this ``reduction to constant coefficients" method 
which is peculiar of $1$-dimensional problems. 
Furthermore we think that this proof, apart from being 
more general, is also simpler than 
the one given in \cite{Feola-Iandoli-Loc}.

The literature in the Euclidean space $\mathbb{R}^d$ is wider. 
After the $1$-dimensional result by Poppenberg \cite{Pop1}, 
there have been the pioneering works by 
Kenig-Ponce-Vega \cite{KPV3, KPV2, KPV} in any dimension. 
More recently these results have been improved, 
in terms of regularity of the initial data, 
by Marzuola-Metcalfe-Tataru in \cite{MMT,MMT2,MMT3}. 
We mention also that Chemin-Salort proved in \cite{CS} 
a very low regularity well-posedness for a  particular 
quasi-linear Schr\"odinger equation in $3$ dimensions 
coupled with an elliptic problem.

We make some short comments on the hypotheses we made on the equation. 
As already pointed out, the equation \eqref{QLNLS} is  Hamiltonian. 
This is quite natural to assume when working on compact manifolds. 
On the Euclidean space one could make some milder 
assumptions because one could  use  the smoothing 
properties of the linear flow (proved by Constantin-Saut in \cite{saut}) 
to somewhat compensate the loss of derivatives introduced 
by the non Hamiltonian terms. 
These smoothing properties are not available 
on compact manifolds. Actually there are very 
interesting examples given by Christ in \cite{Cris} of non 
Hamiltonian equations which are ill-posed on the 
circle $\mathbb{S}^1$ and well-posed on $\R$. 
Strictly speaking the Hamiltonian structure is 
not really fundamental  for our method.
For instance we could consider the not necessarily Hamiltonian 
nonlinearity 
\[
P(u)=g(u)\Delta u+\ii f(u)\cdot \nabla u+h(u)\,,
\]
where $g: \mathbb{C}\to \mathbb{R}$, $f: \mathbb{C}\to \mathbb{R}^d$, 
$h : \mathbb{C}\to \mathbb{C}$ are smooth functions with a zero of order at least $2$. 
Our method would cover also this case.
 We did not insist on this fact because the equation 
 above is \emph{morally} Hamiltonian at the positive orders, 
 in the sense that $f$ and $g$ are not linked, 
 as in an Hamiltonian equation, 
 but they enjoy the same \emph{reality} properties 
 of an Hamiltonian equation. \\
\noindent The Hypothesis \ref{hyp1} is needed 
in order to cover the case of large initial conditions, 
this is compatible with the \emph{global ellipticity} 
condition we assumed in \cite{Feola-Iandoli-Loc} and 
with the ones given in \cite{KPV2,MMT3}. As already said, this hypothesis is 
not necessary in the case of small data.

We discuss briefly the strategy of our proof. 
We begin by performing  a para-linearization of 
the equation \emph{à la Bony} \cite{bony} 
with respect to the variables $(u,\bar{u})$.
 Then, in the 
same spirit of \cite{Feola-Iandoli-Loc},
we construct  the solutions of our problem by means of a 
quasi-linear iterative scheme \emph{\`a la Kato} \cite{Kato-loc}. 
More precisely, starting from the para-linearized system, 
we build a sequence of linear problems which 
converges to a solution of the para-linearized system 
and hence to a solution of the original equation \eqref{QLNLS}. 
 At each step of the iteration one needs to 
solve a linear para-differential system, in the variable $(u,\bar{u})$, 
with non constant coefficients 
(see for instance \eqref{QLNLSV2020}).
This strategy is classical, one can find it for instance in the book by Metivier \cite{Metivier}
for semilinear Schr\"odinger equations. Being our equation quasi-linear, we need some extra steps
(w.r.t. \cite{Metivier})
in order to
prove the existence of the solutions of the linear paradifferential system. 
We provide \emph{a priori} energy estimates (see Theorem \ref{StimeEnergia}).
In order to do this, we decouple the equations for  $(u,\bar{u})$ up to semilinear terms.
This is done by applying changes of coordinates generated by 
para-differential operators. Once achieved such a diagonalization we are able to prove energy estimates in an \emph{energy-norm}, which is equivalent to the Sobolev one.
In constructing such energy norm, we need to introduce a 
microlocal cut-off (see \eqref{cutofftheta}), 
relating the amplitude of the initial condition
and the 
frequency of the solution, which allows us to obtain the result without
any smallness assumption.

The paper is organized as follows. In Section \ref{sec:2} 
we give a short and self-contained introduction to 
the para-differential calculus that is needed in the rest of the paper. 
In Section \ref{sec:3} we perform the 
para-linearization of the equation. 
In Section \ref{sec:4} we give an \emph{a priori} 
energy estimate on the linearized equation by 
performing suitable changes of coordinates. 
In Section \ref{sec:5} we give the proof of Theorem \ref{main}. 
\bigskip

\section{Functional setting}\label{sec:2}
We denote by $H^{s}(\mathbb{T}^d;\mathbb{C})$
(respectively $H^{s}(\mathbb{T}^d;\mathbb{C}^{2})$)
the usual Sobolev space of functions $\mathbb{T}^{d}\ni x \mapsto u(x)\in \mathbb{C}$
(resp. $\C^{2}$).
We expand a function $ u(x) $, $x\in \mathbb{T}^{d}$, 
 in Fourier series as 
\be\label{complex-uU}
u(x) = \frac{1}{(2\pi)^{{d}/{2}}}
\sum_{n \in \Z^{d} } \hat{u}(n)e^{\ii n\cdot x } \, , \qquad 
\hat{u}(n) := \frac{1}{(2\pi)^{{d}/{2}}} \int_{\mathbb{T}^{d}} u(x) e^{-\ii n \cdot x } \, dx \, .
\ee
We also use the notation
\begin{equation}\label{notaFou}
u_n^+ := u_n := \hat{u}(n) \qquad  {\rm and} 
\qquad  u_n^- := \ov{u_n}  := \ov{\hat{u}(n)} \, . 
\end{equation}
We set $\langle j \rangle:=\sqrt{1+|j|^{2}}$ for $j\in \mathbb{Z}^{d}$.
We endow $H^{s}(\mathbb{T}^{d};\mathbb{C})$ with the norm 
\begin{equation}\label{Sobnorm}
\|u(\cdot)\|_{H^{s}}^{2}:=\sum_{j\in \mathbb{Z}^{d}}\langle j\rangle^{2s}|u_{j}|^{2}\,.
\end{equation}
For $U=(u_1,u_2)\in H^{s}(\mathbb{T}^d;\mathbb{C}^{2})$
we just set
$\|U\|_{H^{s}}=\|u_1\|_{H^{s}}+\|u_2\|_{H^{s}}$.
Moreover, for $r\in\R^{+}$, we 
denote by $B_{r}(H^{s}(\mathbb{T}^{d};\mathbb{C}))$
(resp. $B_{r}(H^{s}(\mathbb{T}^{d};\mathbb{C}^2))$)
the ball of $H^{s}(\mathbb{T}^{d};\mathbb{C})$ 
(resp. $H^{s}(\mathbb{T}^{d};\mathbb{C}^2)$)
with radius $r$
centered at the origin.
We shall also write the norm in \eqref{Sobnorm} as
\begin{equation}\label{Sobnorm2}
\|u\|^{2}_{H^{s}}=(\langle D\rangle^{s}u,\langle D\rangle^{s} u)_{L^{2}}\,, 
\qquad
\langle D\rangle e^{\ii j\cdot x}=\langle j\rangle  e^{\ii j\cdot x}\,,\;\;\; 
\forall \, j\in \mathbb{Z}^{d}\,,
\end{equation}
where $(\cdot,\cdot)_{L^{2}}$ denotes the standard complex $L^{2}$-scalar product
 \begin{equation}\label{scalarL}
 (u,v)_{L^{2}}:=\int_{\mathbb{T}^{d}}u\bar{v}dx\,, 
 \qquad \forall\, u,v\in L^{2}(\mathbb{T}^{d};\mathbb{C})\,.
 \end{equation}
  \noindent
{\bf Notation}. 
We shall 
use the notation $A\lesssim B$ to denote 
$A\le C B$ where $C$ is a positive constant
depending on  parameters fixed once for all, for instance $d$
 and $s$.
 We will emphasize by writing $\lesssim_{q}$
 when the constant $C$ depends on some other parameter $q$.

\subsection{Basic Para-differential calculus}\label{basicPara}
We introduce the   symbols we shall use in this paper.
We shall consider symbols 
$\mathbb{T}^{d}\times \mathbb{R}^{d}\ni (x,\x)\to a(x,\x)$
 in the spaces 
$\mathcal{N}_{s}^{m}$, $m,s\in \mathbb{R}$, $s\geq0$,
%
%
%
defined by the 
norms
\begin{equation}\label{normaSimbo}
|a|_{\mathcal{N}_{s}^{m}}:=\sup_{|\alpha|+|\beta|\leq s}\sup_{ \x\in \mathbb{R}^{d}}
\langle \x\rangle^{-m+|\beta|}\|\pa_{\x}^{\beta}\pa_{x}^{\alpha}a(x,\x)\|_{L^{\infty}}\,.
\end{equation}
The constant $m\in \mathbb{R}$ indicates the \emph{order} of the symbols, while
$s$ denotes its differentiability.
Let $0<\epsilon< 1/4$, consider
a smooth function $\chi : \mathbb{R}\to[0,1]$ satisfying 
\begin{equation}\label{cutofffunct}
\chi(\x)=\left\{
\begin{aligned}
&1 \quad {\rm if} |\x|\leq 5/4 
\\
&0 \quad {\rm if} |\x|\geq 8/5 
\end{aligned}\right.
\end{equation}
and define
\begin{equation}\label{cutofffunctepsilon}
\chi_{\epsilon}(\x):=\chi(|\x|/\epsilon)\,.
\end{equation}

 \noindent
 For a symbol $a(x,\x)$ in $\mathcal{N}_{s}^{m}$
we define its (Weyl)
quantization as 
\begin{equation}\label{quantiWeyl}
T_{a}h:=\frac{1}{(2\pi)^{d}}\sum_{j\in \mathbb{Z}^{d}}e^{\ii j\cdot x}
\sum_{k\in\mathbb{Z}^{d}}
\chi_{\epsilon}\Big(\frac{|j-k|}{\langle j+k\rangle}\Big)
\widehat{a}\big(j-k,\frac{j+k}{2}\big)\widehat{h}(k)
\end{equation}
where $\widehat{a}(\eta,\x)$ denotes the $\eta$-Fourier coefficient 
of $a(x,\x)$ in  the variable 
$x\in \mathbb{T}^{d}$.
\begin{remark}\label{positivity}
 The definition of the operator $T_a$ is independent of the choice of the  cut-off function $\chi_{\epsilon}$ 
 up to regularizing operators (satisfying  estimates of the form \eqref{diffQuanti}),  
 this will be one of the consequences of  Lemma \ref{azioneSimboo}.
\end{remark}
 
 \begin{remark}\label{strutAlg}
Let us consider a symbol $a(x,\x)$ of order $m$ and set $A:=T_{a}$. 
Then one can check the following:
\begin{align}
 \bar{A}[h]&:=\ov{A[\bar{h}]}\,,\quad  \Rightarrow\quad 
\bar{A}=T_{\tilde{a}}\,,\qquad \tilde{a}(x,\x)=\ov{a(x,-\x)}\,;\label{simboBarrato}\\
{\bf (Adjoint)} \;\;  (Ah , v)_{L^{2}}&=:(h,A^{*}v)_{L^{2}}\,, \quad \Rightarrow \quad 
A^{*}=T_{\ov{a}}\,.\label{simboAggiunto}
\end{align}
If the symbol $a$ is real valued then the operator $T_{a}$ is self-adjoint
with respect to the scalar product in \eqref{scalarL}.
\end{remark}

\vspace{0.3em}
\noindent
{\bf Notation.}
 Given a symbol $a(x,\x)$ we shall also write
 \begin{equation}\label{notaPara}
 T_{a}[\cdot]:=\opbw(a(x,\x))[\cdot]\,,
 \end{equation}
 to denote the associated para-differential operator.

We now recall some fundamental properties of para-differential operators.

\begin{lemma}\label{azioneSimboo}
The following holds.

\noindent
$(i)$ Let $m_1,m_2\in \mathbb{R}$, $s>d/2$, $s\in \mathbb{N}$ and $a\in\mathcal{N}^{m_1}_s$, 
$b\in \mathcal{N}^{m_2}_s$. One has
\begin{equation}\label{prodSimboli}
|ab|_{\mathcal{N}^{m_1+m_2}_s}+|\{a,b\}|_{\mathcal{N}_{s-1}^{m_1+m_2-1}}+
|\s(a,b)|_{\mathcal{N}_{s-2}^{m_1+m_2-2}}\lesssim
|a|_{\mathcal{N}_{s}^{m_1}}|b|_{\mathcal{N}_{s}^{m_2}}
\end{equation}
where
\begin{equation}\label{PoissonBra}
\{a,b\}:=\sum_{j=1}^{d}\Big((\pa_{\x_j}a)(\pa_{x_{j}}b)
-(\pa_{x_j}a)(\pa_{\x_{j}}b)\Big)\,,
\end{equation}
\begin{equation}\label{PoissonBra22}
\s(a,b):=\sum_{j,k=1}^{d}\Big((\pa_{\x_{j}\x_{k}}a)(\pa_{x_jx_{k}}b)
-2(\pa_{x_{j}\x_{k}}a)(\pa_{\x_j x_{k}}b)
+(\pa_{x_{j}x_{k}}a)(\pa_{\x_j\x_{k}}b)\Big)\,.
\end{equation}

\noindent
$(ii)$ Let $s_0>d$, $s_0\in \mathbb{N}$, $m\in \mathbb{R}$ 
and $a\in\mathcal{N}_{s_0}^{m}$. 
Then, for any $s\in \mathbb{R}$, one has
\begin{equation}\label{actionSob}
\|T_{a}h\|_{H^{s-m}}\lesssim|a|_{\mathcal{N}^{m}_{s_0}}\|h\|_{H^{s}}\,,
\qquad \forall h\in H^{s}(\mathbb{T}^{d};\mathbb{C})\,.
\end{equation}

\noindent
$(iii)$ Let $s_0>d$, $s_0\in \mathbb{N}$, $m\in \mathbb{R}$, $\rho\in\mathbb{N}$,  and 
$a\in\mathcal{N}_{s_0+\rho}^{m}$.
For $0<\epsilon_2\leq \epsilon_1<1/2$ and any 
$h\in H^{s}(\mathbb{T}^{d};\mathbb{C})$, 
we define
\begin{equation}\label{natale}
R_{a}h:=\frac{1}{(2\pi)^{d}}\sum_{j\in \mathbb{Z}^{d}}e^{\ii j\cdot x}
\sum_{k\in\mathbb{Z}^{d}}
\big(\chi_{\epsilon_1}-\chi_{\epsilon_2}\big)\Big(\frac{|j-k|}{\langle j+k\rangle}\Big)
\widehat{a}(j-k,\frac{j+k}{2})\widehat{h}(k)\,,
\end{equation}
where $\chi_{\epsilon_1}, \chi_{\epsilon_2}$ are as in \eqref{cutofffunctepsilon}. 
Then one has
 \begin{equation}\label{diffQuanti}
 \|R_a h\|_{H^{s+\rho-m}}\lesssim
 \|h\|_{H^{s}}|a|_{\mathcal{N}^{m}_{\rho+s_0}}\,,
 \qquad \forall h\in H^{s}(\mathbb{T}^{d};\mathbb{C})\,.
 \end{equation}
 
 \noindent
 $(iv)$ 
 Let $s_0>d$, $s_0\in \mathbb{N}$, $m\in \mathbb{R}$ and $a\in\mathcal{N}_{s_0}^{m}$. 
 For $\mathtt{R}>0$, consider 
 the cut-off function $\mathcal{X}_{\mathtt{R}}\in C^{\infty}(\mathbb{R}^{n};\mathbb{R})$
 defined as
 \begin{equation}\label{cutofftheta}
 \mathcal{X}_{\mathtt{R}}(\x):=1-\chi\Big(\frac{|\x|}{\mathtt{R}}\Big)\,,
 \end{equation}
 where $\chi$ is given in \eqref{cutofffunct}
 and define 
 the symbol 
 $a^{\perp}_{\mathtt{R}}(x,\x):=(1-\mathcal{X}_{\mathtt{R}}(\x))a(x,\x)$. 
 Then, for any $q\in \mathbb{R}$ such that $q+m\geq 0$, 
 one has
 \begin{equation}\label{azionecufoffoff}
 \|T_{a^{\perp}_{\mathtt{R}}} h\|_{H^{s+q}}
 \lesssim_{q,m} \mathtt{R}^{q+m} \|h\|_{H^{s}}|a|_{\mathcal{N}_{s_0}^{m}}\,,
 \qquad \forall\, h\in H^{s}(\mathbb{T}^{d};\mathbb{C})\,.
 \end{equation}
 \end{lemma}
 
 \begin{proof}
$(i)$
 For any $|\alpha|+|\beta|\leq s$ we have
 \[
 \pa_{x}^{\alpha}\pa_{\x}^{\beta}\Big(a(x,\x) b(x,\x)\Big)=
 \sum_{\substack{\alpha_1+\alpha_2=\alpha \\ 
\beta_1+\beta_2=\beta }}C_{\alpha,\beta}
( \pa_{x}^{\alpha_1}\pa_{\x}^{\beta_1}a)(x,\x)
( \pa_{x}^{\alpha_2}\pa_{\x}^{\beta_2}b)(x,\x)
 \]
 for some combinatoric coefficients $C_{\alpha,\beta}>0$.
 Then, recalling \eqref{normaSimbo},
 \[
 \|( \pa_{x}^{\alpha_1}\pa_{\x}^{\beta_1}a)(x,\x)
( \pa_{x}^{\alpha_2}\pa_{\x}^{\beta_2}b)(x,\x)\|_{L^{\infty}}
\lesssim_{\alpha,\beta} 
|a|_{\mathcal{N}_s^{m_1}}
|b|_{\mathcal{N}_s^{m_2}}\langle \x\rangle^{m_1+m_2-|\beta|}\,.
 \]
This implies the \eqref{prodSimboli} for the product $ab$. 
 The \eqref{prodSimboli}
for the symbols 
$\{a,b\}$ and $\s(a,b)$ follows similarly 
using \eqref{PoissonBra} and \eqref{PoissonBra22}.

\smallskip
\noindent $(ii)$ 
First of all notice that, since $a\in \mathcal{N}_{s_0}^{m}$, 
$s_0\in\mathbb{N}$, then
(recall \eqref{normaSimbo})
\[
\|a(\cdot,\x)\|_{H^{s_0}}\lesssim \langle\x\rangle^{m}|a|_{\mathcal{N}_{s_0}^{m}}\,,\;\;\forall\x\in \mathbb{Z}^{d}\,, 
\]
which implies 
\begin{equation}\label{virus5}
|\hat{a}(j,\x)|\lesssim \langle\x\rangle^{m}|a|_{\mathcal{N}_{s_0}^{m}}\langle j\rangle^{-s_0}\,,
\quad \forall\, j,\x\in \mathbb{Z}^{d}\,.
\end{equation}
Moreover, since $0<\epsilon<1/4$ we note that, for $\x,\eta\in \mathbb{Z}^d$,
\begin{equation}\label{equixieta}
\chi_{\epsilon}\left(\frac{|\x-\eta|}{\langle \x+\eta\rangle}\right)\neq0\quad \Rightarrow\quad
\left\{
\begin{aligned}
&(1-\tilde\epsilon)|\x|\leq (1+\tilde\epsilon)|\eta| \\
&(1-\tilde\epsilon)|\eta|\leq (1+\tilde\epsilon)|\x| \,,
\end{aligned}
\right.
\end{equation}
where $0<\tilde{\epsilon}<4/5$. Indeed, recalling \eqref{cutofffunct}-\eqref{cutofffunctepsilon},
we have $|\x|\leq (1+\tfrac{8}{5}\epsilon)|\eta|+\tfrac{8}{5}\epsilon|\x|+\tfrac{8}{5}\epsilon$
which, for $|\eta|\neq0$, implies
$(1-\tfrac{8}{5}\epsilon)|\x|\leq (1+\tfrac{16}{5}\epsilon)|\eta|$.
This implies the first condition in \eqref{equixieta} in the case $|\eta|\neq0$.
The case $|\eta|=0$ is trivial since the definition of the cut-off function $\chi_{\e}$ and $\e<1/4$
implies that $|\x|=0$ as well.
The second condition in \eqref{equixieta} is similar.
As a consequence we have
 $\langle\x+\eta\rangle\sim\langle\x\rangle $: on one hand
 $|\x+\eta|\leq |\x|+|\eta|\leq (1+C)|\x|$ for some $C=C(\tilde{\epsilon})>0$; 
 on the other hand $|\x|=\tfrac{1}{2}|\x-\eta+\x+\eta|\leq \tfrac{1}{2}|\x-\eta|
 +\tfrac{1}{2}|\x+\eta|\leq \tfrac{1}{2}\tfrac{8\epsilon}{5}\langle\x+\eta\rangle+\tfrac{1}{2}\langle\x+\eta\rangle$.
Therefore, taking $s_0>d$,
\begin{equation}\label{natale2}
\begin{split}
\|T_{a}h\|^{2}_{H^{s-m}}&
\stackrel{\mathclap{\eqref{Sobnorm}}}{\lesssim}
\sum_{\x\in\mathbb{Z}^{d}}
\langle\x\rangle^{2(s-m)}\Big|\sum_{\eta\in \mathbb{Z}^{d}}
\chi_{\epsilon}\left(\frac{|\x-\eta|}{\langle \x+\eta\rangle}\right)\hat{a}(\x-\eta, \frac{\x+\eta}{2})\hat{h}(\eta)\Big|^{2}\\
&\stackrel{\mathclap{\eqref{virus5}, \eqref{equixieta}}}{\lesssim}\,\,\,\,\,\,
\sum_{\x\in \mathbb{Z}^{d}}\langle\x\rangle^{-2m}\Big(\sum_{\eta\in \mathbb{Z}^{d}}
\frac{\langle\x\rangle^{m}}{\langle\x-\eta\rangle^{s_0}}
|\hat{h}(\eta)|\langle \eta\rangle^{s}
\Big)^{2} |a|_{\mathcal{N}_{s_0}^{m}}^{2}\\
&\lesssim|a|^{2}_{\mathcal{N}^{m}_{s_0}} \sum_{\xi\in\mathbb{Z}^d}\Big(\sum_{\eta\in\mathbb{Z}^d}|\hat{h}(\eta)\langle\eta\rangle^s\frac{1}{\langle\xi-\eta\rangle^{s_0}}|\Big)^2\\
&
\lesssim|a|^{2}_{\mathcal{N}^{m}_{s_0}}\|\hat{h}(\xi)\langle\xi\rangle^s\star{\langle\xi\rangle}^{-s_0}\|_{\ell^2(\mathbb{Z}^d)}^2\leq |a|^{2}_{\mathcal{N}^{m}_{s_0}}\|\hat{h}(\xi)\langle\xi\rangle^s\|_{\ell^2(\mathbb{Z}^d)}^2\|\langle{\xi}\rangle^{-s_0}\|_{\ell^1(\mathbb{Z}^d)}^2\\
&\lesssim
\|h\|_{H^{s}}^{2}|a|^{2}_{\mathcal{N}^{m}_{s_0}}\,,
\end{split}
\end{equation}
where we denoted by $\star$ the convolution between sequences, in the penultimate passage we used the Young inequality for sequences and in the last one that $\langle\xi\rangle^{-s_0}$ is in $\ell^1(\Z^d)$ since $s_0>d$.

\smallskip
\noindent
$(iii)$ Notice that the set of $\x,\eta$ such that
$(\chi_{\epsilon_1}-\chi_{\epsilon_2})(|\x-\eta|/\langle\x+\eta\rangle)=0$
contains the set  such that
\[
|\x-\eta|\geq \frac{8}{5}\epsilon_1 \langle \x+\eta\rangle\quad {\rm or}\quad
|\x-\eta|\leq \frac{5}{4}\epsilon_2 \langle\x+\eta\rangle\,.
\]
Therefore $(\chi_{\epsilon_1}-\chi_{\epsilon_2})(|\x-\eta|/\langle\x+\eta\rangle)\neq0$
implies
\begin{equation}\label{condizio}
\frac{5}{4}\epsilon_2 \langle\x+\eta\rangle\leq |\x-\eta|\leq \frac{8}{5}\epsilon_1 \langle\x+\eta\rangle\,.
\end{equation}
For $\x\in \mathbb{Z}^{d}$ 
we denote $\mathcal{A}(\x)$ the set of 
$\eta\in \mathbb{Z}^{d}$ such that the \eqref{condizio}
holds. Moreover (reasoning as in \eqref{virus5}), since
$a\in \mathcal{N}_{s_0+\rho}^{m}$, we have that
\begin{equation}\label{virus6}
|\hat{a}(j,\x)|\lesssim \langle\x\rangle^{m}|a|_{\mathcal{N}_{s_0+\rho}^{m}}
\langle j\rangle^{-s_0-\rho}\,,
\quad \forall\, j,\x\in \mathbb{Z}^{d}\,.
\end{equation}
To estimate the remainder in \eqref{natale}
we reason as in \eqref{natale2}.
By \eqref{condizio} and setting $\rho=s-s_0$ we have
\begin{equation}\label{stimarestoresto}
\begin{aligned}
\|R_ah\|_{H^{s+\rho-m}}^{2}
&\stackrel{\mathclap{\eqref{Sobnorm}}}{\lesssim}\sum_{\x\in\mathbb{Z}^{d}}
\langle\x\rangle^{2(s+\rho-m)}\Big| (\chi_{\epsilon_1}-\chi_{\epsilon_2})\left(\frac{|\x-\eta|}{\langle\x+\eta\rangle}\right)\hat{a}(\x-\eta, \frac{\x+\eta}{2})\hat{h}(\eta)\Big|^{2}\\
&\stackrel{\mathclap{\eqref{virus6}}}{\lesssim}
\sum_{\x\in \mathbb{Z}^{d}}
\langle \x\rangle^{-2m}
\Big(\sum_{\eta\in \mathcal{A}(\x)}
\frac{\langle\x-\eta\rangle^{\rho}\langle\x+\eta\rangle^{m}}{\langle\x-\eta\rangle^{\rho+s_0}}
|\hat{h}(\eta)|\langle \eta\rangle^{s}
\Big)^{2}|a|_{\mathcal{N}^{m}_{s_0+\rho}}^{2}\\
&\lesssim \|\hat{h}(\xi)\langle\xi\rangle^s\star\langle\xi\rangle^{-s_0}\|^2_{\ell^2(\mathbb{Z}^d)}|a|^{2}_{\mathcal{N}^{m}_{\rho+s_0}}\lesssim \|\hat{h}(\xi)\langle\xi\rangle^s\|^2_{\ell^2(\mathbb{Z}^d)} \|\langle\xi\rangle^{-s_0}\|^2_{\ell^1(\mathbb{Z}^d)}|a|^{2}_{\mathcal{N}^{m}_{\rho+s_0}}\\
&\lesssim
\|h\|_{H^{s}}^{2}|a|^{2}_{\mathcal{N}^{m}_{\rho+s_0}}\,,
\end{aligned}
\end{equation}
where we have denoted by $\star$ the convolution between sequences, in the penultimate step we used Young inequality for sequences, in the last one we used that $\langle\xi\rangle^{-s_0}$ is in $\ell^1(\Z^d)$ since $s_0>d$.

\noindent
 $(iv)$ This item follows by reasoning exactly  as in the proof of  item $(iii)$
and recalling that, by the definition of $\mathcal{X}_{\mathtt{R}}$ in \eqref{cutofftheta},
one has that $a^{\perp}_{\mathtt{R}}(x,\x)\equiv0$ for any $|\x|> 3\mathtt{R}$.
 \end{proof}
 \begin{remark}\label{rmk:regola}
 The estimate \eqref{actionSob} is not optimal. 
 By following the more sophisticated proof by 
 Metivier in \cite{Metivier} one could obtain  the 
 better bound with $|a|_{\mathcal{N}_0^m}$  
 instead of $|a|_{\mathcal{N}_{s_0}^m}$ 
 on the right hand side. 
 \end{remark}
 
 \begin{proposition}{\bf (Composition).}\label{prop:compo}
Fix $s_0>d$, $s_0\in \mathbb{N}$, and $m_1,m_2\in \mathbb{R}$. Then the following holds.

\noindent
$(i)$ For 
 $a\in \mathcal{N}_{s_0+4}^{m_1}$ and $b\in\mathcal{N}_{s_0+4}^{m_2}$
we have (recall \eqref{PoissonBra}, \eqref{PoissonBra22})
\begin{equation}\label{composit}
T_{a}\circ T_{b}=T_{ab}+\frac{1}{2\ii }T_{\{a,b\}}-\frac{1}{8}T_{\s(a,b)}+R(a,b)\,,
\end{equation}
where $R(a,b)$ is a remainder  satisfying, for any $s\in \mathbb{R}$,
\begin{equation}\label{composit2}
\|R(a,b)h\|_{H^{s-m_1-m_2+3}}\lesssim\|h\|_{H^{s}}|a|_{\mathcal{N}^{m_1}_{s_0+4}}
|b|_{\mathcal{N}^{m_2}_{s_0+4}}\,.
\end{equation}
Moreover, if $a,b\in H^{\rho+s_0}(\mathbb{T}^{d};\mathbb{C})$ are functions 
(independent of $\x\in \mathbb{R}^{n}$)
then, $\forall s\in\mathbb{R}$,
\begin{equation}\label{composit3}
\|(T_aT_b-T_{ab})h\|_{H^{s+\rho}}\lesssim\|h\|_{H^{s}}\|a\|_{H^{\rho+s_0}}
\|b\|_{H^{\rho+s_0}}\,.
\end{equation}

\noindent
$(ii)$ Let $a,b$ as in item $(i)$ and, for $\mathtt{R}>0$,  define
$a_{\mathtt{R}}(x,\x):=\mathcal{X}_{\mathtt{R}}(\x)a(x,\x)$, 
$b_{\mathtt{R}}(x,\x):=\mathcal{X}_{\mathtt{R}}(\x)b(x,\x)$
where $\mathcal{X}_{\mathtt{R}}(\x)$ is defined in \eqref{cutofftheta}. 
Assume that $m_1+m_2-2\leq 0$.
Then
\begin{equation}\label{compositTheta}
T_{a_\mathtt{R}}\circ T_{b_\mathtt{R}}=T_{a_{\mathtt{R}}b_{\mathtt{R}}}
+\frac{1}{2\ii }T_{\{a_{\mathtt{R}},b_{\mathtt{R}}\}}
-\frac{1}{8}T_{\s(a_\mathtt{R},b_\mathtt{R})}
+R(a_{\mathtt{R}},b_{\mathtt{R}})\,,
\end{equation}
where $R(a_{\mathtt{R}},b_{\mathtt{R}})$ is a remainder satisfying
\begin{equation}\label{composit2Theta}
\|R(a_{\mathtt{R}},b_{\mathtt{R}})h\|_{H^{s-m_1-m_2+2}}
\lesssim \mathtt{R}^{-1}
\|h\|_{H^{s}}|a|_{\mathcal{N}^{m_1}_{s_0+4}}
|b|_{\mathcal{N}^{m_2}_{s_0+4}}\,.
\end{equation}
\end{proposition}

\begin{remark}
We note that when applying the above proposition 
we \emph{consume} four derivatives on the symbols.
In the diagonalization procedure of Section \ref{sec:4},
we shall apply this proposition several times. This requires a certain smoothness (in $x$) 
on the symbols. Since in our case the symbols will depend on the solution $u$ of \eqref{QLNLS}, 
this  smoothness 
is equivalent,  thanks to the arguments in Section \ref{sec:nonomosimboli},
to the regularity in $x$ of the solutions.
\end{remark}

\begin{proof}[{\bf Proof of Proposition \ref{prop:compo}}]
We start by proving the \eqref{composit3}. 
For $\x,\theta,\eta\in \mathbb{Z}^{d}$ we define
\begin{equation}\label{cutofffunctions}
r_1(\x,\theta,\eta):=\chi_{\epsilon}\left(\frac{|\x-\theta|}{\langle\x+\theta\rangle}\right)
\chi_{\epsilon}\left(\frac{|\theta-\eta|}{\langle\theta+\eta\rangle}\right)\,,
\qquad
r_2(\x,\eta):=\chi_{\epsilon}\left(\frac{|\x-\eta|}{\langle\x+\eta\rangle}\right)\,.
\end{equation}
Recalling \eqref{quantiWeyl} 
and that $a, b$ are functions we have
\begin{equation}\label{achille}
\begin{aligned}
&R_0h:= (T_aT_b-T_{ab})h\,,\\
&\widehat{(R_0h)}(\x)=(2\pi)^{-\frac{3d}{2}}\sum_{\eta,\theta\in \mathbb{Z}^{d}}
\big(r_1(\x,\theta,\eta)-r_2(\x,\eta)\big)
 \widehat{a}(\x-\theta)
\widehat{b}(\theta-\eta)\hat{h}(\eta)\,.
\end{aligned}
\end{equation}
Let us define the sets
\begin{align}
D&:=\Big\{
(\x,\theta,\eta)\in \mathbb{Z}^{3d}\; :\;
r_1(\x,\theta,\eta)-r_2(\x,\eta)=0\Big\}\,,\label{insiemeZero}\\
A&:=\Big\{
(\x,\theta,\eta)\in \mathbb{Z}^{3d}\; :\;
\frac{|\x-\theta|}{\langle\x+\theta\rangle}\leq \frac{5\epsilon}{4}\,,\;\;
\frac{|\x-\eta|}{\langle\x+\eta\rangle}\leq \frac{5\epsilon}{4}\,,\;\;
\frac{|\theta-\eta|}{\langle\theta+\eta\rangle}\leq \frac{5\epsilon}{4}\Big\}\,,
\label{insiemeZero2}\\
B&:=\Big\{
(\x,\theta,\eta)\in \mathbb{Z}^{3d}\; :\;
\frac{|\x-\theta|}{\langle\x+\theta\rangle}\geq \frac{8\epsilon}{5}\,,\;\;
\frac{|\x-\eta|}{\langle\x+\eta\rangle}\geq \frac{8\epsilon}{5}\,,\;\;
\frac{|\theta-\eta|}{\langle\theta+\eta\rangle}\geq \frac{8\epsilon}{5}\Big\}\,.\label{insiemeZero3}
\end{align}
We note that
\[
D\supseteq A\cup B\quad \Rightarrow\quad D^{c}\subseteq A^{c}\cap B^{c}\,.
\]
Let $(\x,\theta,\eta)\in D^{c}$ and assume in particular  that 
$(\x,\theta,\eta)\in{\rm Supp}(r_1):=\ov{\{(\x,\theta,\eta) : r_1\neq0\}}$. 
Then, reasoning as in \eqref{equixieta}, we can note that
\begin{equation}\label{navyseal2}
|\x-\eta|\lesssim \epsilon \langle\x+\eta\rangle\,\quad {\rm and}\quad \langle\x\rangle \sim \langle\eta\rangle.
\end{equation}
Notice also that $(\x,\theta,\eta)\in{\rm Supp}(r_2)$ implies  the \eqref{navyseal2} as well.
The rough idea of the proof is based on the fact that, if $(\x,\theta,\eta)\in D^{c}$, then 
there are \emph{at least three} equivalent frequencies among 
$\x, \x-\theta,\theta-\eta, \eta$, therefore \eqref{achille} 
restricted to  $(\x,\theta,\eta)\in D^{c}$ is a regularizing operator.
We need to estimate
\[
\|R_0h\|_{H^{s+\rho}}^{2}\lesssim\sum_{\x\in\mathbb{Z}^{d}} \Big(
\sum_{\eta,\theta}^{*} |\hat{a}(\x-\theta)||\hat{b}(\theta-\eta)||\hat{h}(\eta)
|\langle\x\rangle^{s+\rho}
\Big)^{2}=I+II+III\,,
\]
where $\sum_{\eta,\theta}^{*} $ denotes the sum over indexes satisfying 
\eqref{navyseal2}, the term $I$ denotes the sum on indexes satisfying also
$|\x-\theta|>c\epsilon |\x|$, $II$
denotes the sum on indexes satisfying also
$|\eta-\theta|>c\epsilon |\eta|$, for some $0<c<1$ and $III$ is defined by difference.
We estimate the term $I$. By using \eqref{navyseal2} and $|\x-\theta|>c\epsilon |\x|$ we get
\begin{equation*}
\begin{aligned}
I&\lesssim\sum_{\xi\in\Z^d}\Big(\sum_{\eta,\theta}^{*}|\hat{a}(\x-\theta)||\hat{b}(\theta-\eta)||\hat{h}(\eta)| \langle \eta\rangle^s\langle\xi-\theta\rangle^{\rho}\Big)^2\\
&\lesssim \||\hat{h}(\xi)| \langle \xi\rangle^s\star|\hat{a}(\xi)|\langle\xi\rangle^{\rho}\star|\hat{b}(\xi)|\|_{\ell^2(\Z^d)}^2\\
&\lesssim \||\hat{h}(\xi)| \langle \xi\rangle^s\|_{\ell^2(\Z^d)}^2\||\hat{a}(\xi)| \langle \xi\rangle^{\rho}\|_{\ell^1(\Z^d)}^2 \||\hat{b}(\xi)|\|_{\ell^1(\Z^d)}^2\\
&\lesssim\|h\|_{H^s}^2\|a\|_{H^{s_0+\rho}}^2\|b\|_{H^{s_0}}^2,
\end{aligned}
\end{equation*}
where in the last inequality we used Cauchy-Schwartz and $s_0>d>d/2$.

Reasoning similarly one obtains  $II\lesssim
\|h\|_{H^{s}}^{2}\|a\|^{2}_{H^{s_0}}\|b\|^{2}_{H^{s_0+\rho}}$.
The sum $III$
is restricted to indexes satisfying 
\eqref{navyseal2} and $|\x-\theta|\leq c\epsilon |\x|$, $|\eta-\theta|\leq c\epsilon |\eta|$.
For $0<c<1$ small enough these restrictions imply that $(\x,\eta,\zeta)\in A$, which 
is a contradiction since $(\x,\eta,\zeta)\in D^{c}\subseteq A^{c}$.

Let us check the \eqref{composit2}. 
We first prove that 
\begin{equation}\label{composit22}
T_{a}\circ T_{b}=T_{ab}+\frac{1}{2\ii }T_{\{a,b\}}+R(a,b)\,,
\qquad 
\|R(a,b)h\|_{H^{s-m_1-m_2+2}}\lesssim\|h\|_{H^{s}}|a|_{\mathcal{N}^{m_1}_{s_0+2}}
|b|_{\mathcal{N}^{m_2}_{s_0+2}}\,.
\end{equation}
First of all
we note that
\begin{align}
\widehat{(T_{a}T_{b}h)}(\x)&=\frac{1}{(\sqrt{2\pi})^{3d}}
\sum_{\eta,\theta\in\mathbb{Z}^{d}}
r_1(\x,\theta,\eta) \widehat{a}\big(\x-\theta,\frac{\x+\theta}{2}\big)
\widehat{b}\big(\theta-\eta,\frac{\theta+\eta}{2}\big)\hat{h}(\eta)\,,\label{def:prodotto1}\\
\widehat{(T_{ab}h)}(\x)&=\frac{1}{(\sqrt{2\pi})^{3d}}
\sum_{\eta,\theta\in\mathbb{Z}^{d}}
r_2(\x,\eta) \widehat{a}\big(\x-\theta,\frac{\x+\eta}{2}\big)
\widehat{b}\big(\theta-\eta,\frac{\x+\eta}{2}\big)\hat{h}(\eta)\,,\label{def:prodotto2}\\
\frac{1}{2\ii}\widehat{(T_{\{a,b\}}h)}(\x)&=
\frac{1}{2\ii (\sqrt{2\pi})^{3d}}
\sum_{\eta,\theta\in\mathbb{Z}^{d}}
r_2(\x,\eta) \widehat{(\pa_{\x}a)}\big(\x-\theta,\frac{\x+\eta}{2}\big)\cdot
\widehat{(\pa_{x}b)}\big(\theta-\eta,\frac{\x+\eta}{2}\big)\hat{h}(\eta)
\label{def:prodotto3}\\
&-
\frac{1}{2\ii (\sqrt{2\pi})^{3d}}
\sum_{\eta,\theta\in\mathbb{Z}^{d}}
r_2(\x,\eta) \widehat{(\pa_{x}a)}\big(\x-\theta,\frac{\x+\eta}{2}\big)\cdot
\widehat{(\pa_{\x}b)}\big(\theta-\eta,\frac{\x+\eta}{2}\big)\hat{h}(\eta)\,.\nonumber
\end{align}
In the formul\ae\, above we used the notation $\pa_{x}=(\pa_{x_1},\ldots,\pa_{x_d})$,
similarly for $\pa_{\x}$.
We remark that we can substitute the cut-off function
$r_2$ in \eqref{def:prodotto2}, \eqref{def:prodotto3} with $r_1$ up to smoothing remainders.
This follows because one can treat the 
cut-off function
$r_1(\x,\theta,\eta)-r_2(\x,\eta)$
as done in the proof of \eqref{composit3}.
Write $\x+\theta=\x+\eta+(\theta-\eta)$.
By Taylor expanding the symbols at $\x+\eta$, 
we have
\begin{align}
\widehat{a}\big(\x-\theta,\frac{\x+\theta}{2}\big)&=
\widehat{a}\big(\x-\theta,\frac{\x+\eta}{2}\big)
+\widehat{(\pa_{\x}a)}\big(\x-\theta,\frac{\x+\eta}{2}\big)
\cdot\frac{\theta-\eta}{2}\label{expsimbo1}\\
&+\frac{1}{4}\sum_{j,k=1}^{d}
\int_{0}^{1}(1-\s)
\widehat{(\pa_{\x_{j}\x_{k}}a)}\big(\x-\theta,\frac{\x+\eta}{2}+
\s\frac{\theta-\eta}{2}\big)
(\theta_{j}-\eta_{j})(\theta_{k}-\eta_{k})
d\s\,.\nonumber
\end{align}
Similarly one obtains 
\begin{align}
\widehat{b}\big(\theta-\eta,\frac{\theta+\eta}{2}\big)&=
\widehat{b}\big(\theta-\eta,\frac{\x+\eta}{2}\big)
+\widehat{(\pa_{\x}b)}\big(\theta-\eta,\frac{\x+\eta}{2}\big)
\cdot\frac{\theta-\x}{2}\label{expsimbo2}\\
&+\frac{1}{4}\sum_{j,k=1}^{d}
\int_{0}^{1}(1-\s)
\widehat{(\pa_{\x_{j}\x_{k}}b)}\big(\theta-\eta,\frac{\x+\eta}{2}+
\s\frac{\theta-\x}{2}\big)(\theta_{j}-\x_{j})(\theta_{k}-\x_{k})d\s\,.\nonumber
\end{align}
By \eqref{expsimbo1}, \eqref{expsimbo2}
we deduce that
\begin{equation}\label{virus}
\begin{aligned}
&T_{a}T_bh-T_{ab}h-\frac{1}{2\ii }T_{\{a,b\}}h=\sum_{p=1}^{6}R_p h\,,
\\&
\widehat{(R_p h)}(\x):=\frac{1}{(\sqrt{2\pi})^{3d}}\sum_{\eta,\theta\in\mathbb{Z}^{d}}
r_{1}(\x,\theta,\eta)
g_{p}(\x,\theta,\eta)\hat{h}(\eta)\,,
\end{aligned}
\end{equation}
where the symbols $g_i$ are defined as
\begin{align}
&g_1:= 
\frac{-1}{4}
\sum_{j,k=1}^{d}\int_{0}^{1}(1-\s)
\widehat{(\pa_{x_{k}x_{j}}a)}\Big(\x-\theta,\frac{\x+\eta}{2}\Big)
\widehat{(\pa_{\x_{k}\x_{j}}b)}\Big(\theta-\eta,\frac{\x+\eta}{2}+
\s\frac{\theta-\x}{2}\Big)d\s\,,
\label{simboGG1}\\
&g_2:= 
\frac{-1}{4}
\sum_{j,k=1}^{d}\int_{0}^{1}(1-\s)
\widehat{(\pa_{\x_{k}\x_{j}}a)}\Big(\x-\theta,\frac{\x+\eta}{2}+\s\frac{\theta-\eta}{2}\Big)
\widehat{(\pa_{x_{k}x_{j}}b)}\Big(\theta-\eta,\frac{\x+\eta}{2}\Big)d\s\,,
\label{simboGG2}\\
&g_3:= 
\frac{1}{4}
\sum_{j,k=1}^{d}
\widehat{(\pa_{x_{j}}\pa_{\x_{k}}a)}\Big(\x-\theta,\frac{\x+\eta}{2}\Big)
\widehat{(\pa_{x_{k}}\pa_{\x_{j}}b)}\Big(\theta-\eta,\frac{\x+\eta}{2}\Big)\,,\label{simboGG3}
\end{align}
\begin{align}
&g_4:= 
\frac{-1}{8\ii}
\sum_{j,k,p=1}^{d}\int_{0}^{1}(1-\s)
\widehat{(\pa_{x_{k}x_{j}\x_{p}}a)}\Big(\x-\theta,\frac{\x+\eta}{2}\Big)
\widehat{(\pa_{x_{p}\x_{k}\x_{j}}b)}\Big(\theta-\eta,\frac{\x+\eta}{2}+
\s\frac{\theta-\x}{2}\Big)d\s\,,
\label{simboGG4}\\
&g_5:= 
\frac{-1}{8\ii}
\sum_{j,k,p=1}^{d}\int_{0}^{1}(1-\s)
\widehat{(\pa_{\x_{k}\x_{j}x_{p}}a)}\Big(\x-\theta,\frac{\x+\eta}{2}
+\s\frac{\theta-\eta}{2}\Big)
\widehat{(\pa_{\x_{p}x_{k}x_{j}}b)}\Big(\theta-\eta,\frac{\x+\eta}{2}\Big)d\s\,,
\label{simboGG5}\\
&g_6:= 
\frac{1}{16}
\sum_{j,k,p,q=1}^{d}\int\int_{0}^{1}(1-\s_1)(1-\s_2)
\widehat{(\pa_{\x_{j}\x_{k}x_{p}x_{q}}a)}\Big(\x-\theta,\frac{\x+\eta}{2}
+\s_1\frac{\theta-\eta}{2}\Big)\,,\nonumber\\
&\qquad\qquad\qquad\qquad\qquad\times
\widehat{(\pa_{\x_{p}\x_{q}x_{j}x_{k}}b)}\Big(\theta-\eta,\frac{\x+\eta}{2}
+\s_{2}\frac{\theta-\x}{2}\Big)d\s_1d\s_2\,.\label{simboGG6}
\end{align}
We prove the estimate \eqref{composit2} on each 
term of the sum in \eqref{virus}.
First of all we note that
$r_1(\x,\theta,\eta)\neq0$ implies that
\begin{equation}\label{virus2} 
(\theta,\eta)
\in\big\{\frac{|\x-\theta|}{\langle\x+\theta\rangle}\leq \frac{8}{5}\epsilon\big\}
\bigcap\big\{\frac{|\theta-\eta|}{\langle\theta+\eta\rangle}\leq \frac{8}{5}\epsilon\big\}=:\mathcal{B}(\x)\,,
\quad \x\in \mathbb{Z}^{d}\,.
\end{equation} 
 Moreover we note that
\begin{equation}\label{virus3}
(\theta,\eta)\in \mathcal{B}(\x)\;\;\;\Rightarrow\;\;\; |\x|\lesssim|\theta|\,,\;\;
|\theta|\lesssim|\eta|\,,\;\;|\eta|\lesssim|\x|\,.
\end{equation}
We now study the term  $R_{3}h$ in \eqref{virus} depending 
on  $g_{3}(\x,\theta,\eta)$ in \eqref{simboGG3}.
We need to bound from above, for any $j,k=1,\ldots,d$, 
the $H^{s-m_1-m_2+2}$-Sobolev norm  (see \eqref{virus2}) of function $F_{j,k}(x)$ whose
$\x$-th Fourier coefficient is 
\begin{equation}\label{virus11}
\begin{aligned}
\hat{F}_{j,k}(\x)&:=
\sum_{(\theta,\eta)\in\mathcal{B}(\x)}
\widehat{(\pa_{x_{j}}\pa_{\x_{k}}a)}\big(\x-\theta,\frac{\x+\eta}{2}\big)
\widehat{(\pa_{x_{k}}\pa_{\x_{j}}b)}\big(\theta-\eta,\frac{\x+\eta}{2}\big)
\hat{h}(\eta)\\
&=
\sum_{\eta\in\mathbb{Z}^{d} } 
\widehat{c_{j,k}}\big(\x-\eta,\frac{\x+\eta}{2}\big)\hat{h}(\eta)\,,
\end{aligned}
\end{equation}
where we have defined
\[
\begin{aligned}
\widehat{c_{j,k}}\big(p,\zeta\big)
&:=\sum_{\ell\in \mathbb{Z}^{d}}
\widehat{(\pa_{x_{j}}\pa_{\x_{k}}a)}\big(p-\ell,\zeta\big)
\widehat{(\pa_{x_{k}}\pa_{\x_{j}}b)}\big(\ell,\zeta\big)
\mathtt{1}_{\mathcal{C}(p,\zeta)}\,,\qquad p,\zeta\in \mathbb{Z}^{d}\,,\\
\mathcal{C}(p,\zeta)&:=
\big\{\ell\in \mathbb{Z}^{d}\,:\,\frac{|p-\ell|}{\langle2\zeta+\ell\rangle}
\leq \frac{8}{5}\epsilon\big\}\bigcap
\big\{\ell\in \mathbb{Z}^{d}\,:\,
\frac{|\ell|}{\langle\ell-p+2\zeta\rangle}\leq \frac{8}{5}\epsilon\big\}
\end{aligned}
\]
and $\mathtt{1}_{\mathcal{C}(p,\zeta)}$ is the 
characteristic function of the set $\mathcal{C}(p,\zeta)$.
Reasoning as in \eqref{virus3}, we can deduce that for $\ell\in \mathcal{C}(p,\zeta)$
one has
\begin{equation}\label{virus12}
|2\zeta|\lesssim \frac{1}{2}|2\zeta+p|\,.
\end{equation}
Indeed $\ell\in \mathcal{C}(p,\zeta) $ implies $(\theta,\eta)\in \mathcal{B}(\x)$
by setting 
\begin{equation}\label{virus13}
2\x=2\zeta+p\,,\quad 2\theta=2\ell+2\zeta-p\,,\quad 2\eta=2\zeta-p\,.
\end{equation}
Hence the \eqref{virus12} follows by \eqref{virus3} by observing  that
$2\zeta=\x+\eta$.
Using that   $a\in \mathcal{N}_{s_0+4}^{m_1}$, 
$b\in\mathcal{N}_{s_0+4}^{m_2}$
and reasoning as in \eqref{virus5}
we deduce
\begin{equation}\label{virus10}
|\hat{c_{j,k}}(p,\zeta)|\lesssim \langle\zeta\rangle^{m_1+m_2-2}\langle p\rangle^{-s_0}
|a|_{\mathcal{N}^{m_1}_{s_0+2}}|b|_{\mathcal{N}^{m_2}_{s_0+2}}\,.
\end{equation}
By \eqref{virus11}, \eqref{virus3},  \eqref{Sobnorm}, we get
\[
\begin{aligned}
\|F_{j,k}\|_{H^{s-m_1-m_2+2}}^{2}&\lesssim
\sum_{\x\in\mathbb{Z}^{d}}
\langle \x\rangle^{-2m_1-2m_2+2}\Big(
\sum_{\eta\in \mathbb{Z}^{d}}
|\hat{c_{j,k}}\big(\x-\eta,\frac{\x+\eta}{2}\big)|
|\hat{h}(\eta)|\langle\eta\rangle^{s}\Big)^{2}\\
&\stackrel{\mathclap{\eqref{virus10}, \eqref{virus12}, \eqref{virus13}}}{\lesssim}
|a|^{2}_{\mathcal{N}^{m_1}_{s_0+2}}|b|^{2}_{\mathcal{N}^{m_2}_{s_0+2}}
\sum_{\xi\in\mathbb{Z}^{d}}\Big(\sum_{\eta\in \mathbb{Z}^{d}}
|\hat{h}(\eta)|\langle\eta\rangle^{s}
\frac{1}{\langle \x-\eta\rangle^{s_0}}\Big)^2\\
&\lesssim|a|^{2}_{\mathcal{N}^{m_1}_{s_0+2}}
|b|^{2}_{\mathcal{N}^{m_2}_{s_0+2}}\||\hat{h}(\xi)|\langle\xi\rangle^{s}\star\langle\xi\rangle^{-s_0}\|_{\ell^{2}(\mathbb{Z}^d)}\\
&\lesssim \|h\|_{H^{s}}^{2}
|a|^{2}_{\mathcal{N}^{m_1}_{s_0+2}}
|b|^{2}_{\mathcal{N}^{m_2}_{s_0+2}}\,,
\end{aligned}
\]
where in the last step we used Young inequality for sequences, the Cauchy-Schwartz inequality and that $\langle\xi\rangle^{-s_0}$ is in $\ell^1(\mathbb{Z}^d)$ if $s_0>d$.
Since the estimate above holds for any $j,k=1,\ldots,d$, we deduce the
\eqref{composit22} for the remainder $R_{3}h$ in \eqref{virus}.
By reasoning in the same way one can show that
 the remainders depending on $g_1,g_2$ in 
\eqref{simboGG1}, \eqref{simboGG2} satisfy the bound in \eqref{composit22}
and that
the remainders 
$R_{p}h$ with $p=4,5,6$,
satisfy the \eqref{composit2}.
In order to obtain the expansion \eqref{composit} one can  note that
(see \eqref{simboGG1})
\begin{align}
g_1&=-
\frac{1}{8}
\sum_{j,k=1}^{d}
\widehat{(\pa_{x_{k}x_{j}}a)}\big(\x-\theta,\frac{\x+\eta}{2}\big)
\widehat{(\pa_{\x_{k}\x_{j}}b)}\big(\theta-\eta,\frac{\x+\eta}{2}\big)\label{virus4}\\
&-
\frac{1}{16}
\sum_{j,k,p=1}^{d}\int_{0}^{1}(1-\s)^{2}
\widehat{(\pa_{x_{k}x_{j}}a)}\big(\x-\theta,\frac{\x+\eta}{2}\big)
\widehat{(\pa_{\x_{k}\x_{j}\x_{p}}b)}\big(\theta-\eta,\frac{\x+\eta}{2}+
\s\frac{\theta-\x}{2}\big)(\theta_{p}-\x_{p})d\s\,,\nonumber
\end{align}
here we have used the identity 
$
\int_{0}^{1}f(\s)(1-\s)d\s=\tfrac{1}{2}f(0)+\tfrac{1}{2}
\int_{0}^{1}f'(\tau)(1-\tau)^{2}d\tau\,,
$
which follows from the fundamental theorem of calculus.
Expanding similarly the term $g_2$ in \eqref{simboGG2} and recalling the formula
\eqref{PoissonBra22} one gets the \eqref{composit}.
The estimate for the operator associated to the second summand in \eqref{virus4}
follows by
reasoning as done for the term in \eqref{virus11}.
This concludes the proof of item $(i)$.
Item $(ii)$ follows by reasoning as before 
on the symbols $a_{\mathtt{R}},b_{\mathtt{R}}$.
Notice that the remainder 
$R(a_{\mathtt{R}}, b_{\mathtt{R}})$ (see \eqref{composit2})
maps $H^{s}$ to $H^{s-m_1-m_2+3}$. 
Actually using
 that $a_{\mathtt{R}}\equiv b_{\mathtt{R}}\equiv 0$ 
 if $|\x|\leq 3\mathtt{R}$ one gets the \eqref{composit2Theta}.
\end{proof}

\begin{lemma}{\bf (Paraproduct).}\label{lem:paraproduct}
Fix $s_0>d/2$ and 
let $f,g\in H^{s}(\mathbb{T};\mathbb{C})$ for $s\geq s_0$. Then
\begin{equation}\label{eq:paraproduct}
fg=T_{f}g+T_{g}f+\mathcal{R}(f,g)\,,
\end{equation}
where
\begin{equation}\label{eq:paraproduct2}
\widehat{\mathcal{R}(f,g)}(\x)=\frac{1}{(2\pi)^{d}}
\sum_{\eta\in \mathbb{Z}^{d}}
a(\x-\eta,\xi)\hat{f}(\x-\eta)\hat{g}(\eta)\,,
\qquad |a(v,w)|\lesssim\frac{(1+\min(|v|,|w|))^{\rho}}{(1+\max(|v|,|w|))^{\rho}}\,,
\end{equation}
for any $\rho\geq0$.
For $0\leq \rho\leq s-s_0$ one has
\begin{equation}\label{eq:paraproduct22}
\|\mathcal{R}(f,g)\|_{H^{s+\rho}}\lesssim\|f\|_{H^{s}}\|g\|_{H^{s}}\,.
\end{equation}
\end{lemma}

\begin{proof}
Notice that
\begin{equation}\label{prodFG}
\widehat{(fg)}(\x)=\sum_{\eta\in\mathbb{Z}^{d}}\hat{f}(\x-\eta)\hat{g}(\eta)\,.
\end{equation}
Consider the cut-off function $\chi_{\epsilon}$ defined in \eqref{cutofffunctepsilon}
and define a new cut-off function $\Theta : \mathbb{R} \to [0,1]$ as
\begin{equation}\label{cutoffTHETA}
1=\chi_{\epsilon}\left(\frac{| \x-\eta|}{\langle\x+\eta\rangle}\right)+
\chi_{\epsilon}\left(\frac{|\eta|}{\langle2\x-\eta\rangle}\right)+\Theta(\x,\eta)\,.
\end{equation}
Recalling \eqref{prodFG} and \eqref{quantiWeyl} we note that 
\begin{equation}\label{paraprodFG}
\widehat{(T_fg)}(\x)=\sum_{\eta\in\mathbb{Z}^{d}}
\chi_{\epsilon}\left(\frac{|\x-\eta|}{\langle\x+\eta\rangle}\right)
\hat{f}(\x-\eta)\hat{g}(\eta)\,, \quad
\widehat{(T_g f)}(\x)=\sum_{\eta\in\mathbb{Z}^{d}}
\chi_{\epsilon}\left(\frac{|\eta|}{\langle2\x-\eta\rangle}\right)
\hat{f}(\x-\eta)\hat{g}(\eta)\,, 
\end{equation}
and 
\begin{equation}\label{pararestoFG}
\mathcal{R}:=\mathcal{R}(f,g)\,,\qquad
\widehat{\mathcal{R}}(\x):=
\sum_{\eta\in\mathbb{Z}^{d}}
\Theta(\x,\eta)
\hat{f}(\x-\eta)\hat{g}(\eta)\,.
\end{equation}
To obtain the second in \eqref{paraprodFG} one has to use the \eqref{quantiWeyl}
and perform the change of variable $\x-\eta\rightsquigarrow \eta$.
By the definition of the cut-off function $\Theta(\x,\eta)$ we deduce
that, if $\Theta(\x,\eta)\neq0$ we must have
\begin{equation}\label{condTheta}
|\x-\eta|\geq \frac{5\epsilon}{4}\langle\x+\eta\rangle\quad {\rm and}
\quad 
|\eta|\geq \frac{5\epsilon}{4}\langle2\x-\eta\rangle
\qquad \Rightarrow\quad
\langle \eta\rangle \sim \langle \x-\eta\rangle\,.
\end{equation}
This implies that, setting 
$a(\x-\eta,\eta):=\Theta(\x,\eta)$, we get the \eqref{eq:paraproduct2}.
The \eqref{condTheta} also implies that
$\langle\x\rangle\lesssim
\max\{
\langle \x-\eta\rangle,\langle\eta\rangle\}$.
Then we have
\[
\begin{aligned}
\|\mathcal{R}h\|_{H^{s+\rho}}^{2}&\lesssim\sum_{\x\in\mathbb{Z}^{d}}\Big(
\sum_{\eta\in\mathbb{Z}^{d}}
|a(\x-\eta,\eta)||\hat{f}(\x-\eta)||\hat{g}(\eta)|
\langle\x\rangle^{s+\rho}
\Big)^{2}\\
&\stackrel{\mathclap{\eqref{eq:paraproduct2}}}{\lesssim}
\sum_{\x\in\mathbb{Z}^{d}}\Big(
\sum_{\langle\x-\eta\rangle\geq \langle\eta\rangle}
\langle\x-\eta\rangle^{s}|\hat{f}(\x-\eta)| \langle \eta\rangle^{\rho}|\hat{g}(\eta)|
\Big)^{2}
\\&+
\sum_{\x\in\mathbb{Z}^{d}}\Big(
\sum_{\langle\x-\eta\rangle\leq \langle\eta\rangle}
\langle \x-\eta\rangle^{\rho}|\hat{f}(\x-\eta)||\hat{g}(\eta)|
\langle\eta\rangle^{s}
\Big)^{2}\\
&\lesssim
\sum_{\x,\eta\in\mathbb{Z}^{d}}
\langle \eta\rangle^{2(s_0+\rho)}|\hat{g}(\eta)|^{2}
\langle\x-\eta\rangle^{2s}|\hat{f}(\x-\eta)|^{2}\\
&+\sum_{\x,\eta\in\mathbb{Z}^{d}}
\langle \eta\rangle^{2s}|\hat{g}(\eta)|^{2}
\langle\x-\eta\rangle^{2(s_0+\rho)}|\hat{f}(\x-\eta)|^{2}\\
&\lesssim\|f\|_{H^{s}}^{2}\|g\|_{H^{s_0+\rho}}^{2}
+\|f\|_{H^{s_0+\rho}}^{2}\|g\|_{H^{s}}^{2}\,,
\end{aligned}
\]
which implies the \eqref{eq:paraproduct22} for  $s_0+\rho\leq s$.
\end{proof}

\subsection{Real-to-real, Self-adjoint operators}
In this section we analyze some algebraic properties of para-differential operators.
Let us consider a linear operator 
\begin{equation}\label{vinello2}
M:=(M_{\s}^{\s'})_{\sigma,\sigma'\in\{\pm\}}:=:\left(
\begin{matrix}
M_{+}^{+} & M_{+}^{-} \vspace{0.2em} \\
M_{-}^{+}  & M_{-}^{-} 
\end{matrix}
\right)\; :\; H^{s+p}(\mathbb{T}^{d};\mathbb{C}^2)\to
H^{s}(\mathbb{T}^{d};\mathbb{C}^2)
\end{equation}
for some $p\in\mathbb{R}$.
We have the following definition.

\begin{definition}{\bf (Real-to-real maps).}\label{riassunto-mappe}
Consider a linear operator $A : H^{s+p}(\mathbb{T}^{d};\mathbb{C})\to
H^{s}(\mathbb{T}^{d};\mathbb{C})$ for some $p\in \mathbb{R}$.
We associate the linear  operator $\ov{A}[\cdot]$ defined by the relation 
\begin{equation}\label{opeBarrato}
\ov{A}[v] := \ov{A[\ov{v}]} \, ,   \quad \forall v \in H^{s+p}(\mathbb{T}^{d};\mathbb{C})\, .
\end{equation}
We say that a matrix $M$ of operators acting in $ \C^2 $
of the form \eqref{vinello2} is  
 \emph{real-to-real}, if it has the form 
\begin{equation}\label{vinello}
M = \big( M_{\s}^{\s'}\big)_{\s,\s'\in\{\pm\}}\,,
\quad M_{\s}^{\s'}=\ov{M_{-\s}^{-\s'}}
\end{equation}
where $\ov{M_{\s}^{\s'}}$ are defined as in \eqref{opeBarrato}.
\end{definition}

\begin{remark}\label{rmk:Real}
Let $\mathfrak{F}$ a matrix of operators as in \eqref{vinello2}.
If $\mathfrak{F}$ is 
 real-to-real (according to Def. \ref{riassunto-mappe})
 then it 
  preserves the subspace
$\mathcal{U}$ defined as
 \begin{equation}\label{spazioUU}
 \mathcal{U}:=\big\{(u^{+},u^{-})\in 
 L^{2}(\mathbb{T}^{d};\mathbb{C})\times L^{2}(\mathbb{T}^{d};\mathbb{C})\; :\; 
 u^{+}=\ov{u^{-}}
 \big\}\,.
 \end{equation}
In particular it has the form (see \eqref{opeBarrato}, \eqref{vinello})
\begin{equation}\label{barrato4}
\mathfrak{F}:=\left(\begin{matrix} A & B \\ \ov{B} & \ov{A} \end{matrix}\right)\,.
\end{equation}
\end{remark}
\noindent
We consider 
 the scalar product on 
$L^{2}(\mathbb{T}^{d};\mathbb{C}^{2})\cap \mathcal{U}$ 
(see \eqref{spazioUU}) given by 
\begin{equation}\label{comsca}
(U,V)_{L^{2}}:=
\int_{\mathbb{T}^{d}}U\cdot \ov{V}dx:=\int_{\mathbb{T}^{d}}(u\bar{v}+\bar{u}v)dx\,, 
\qquad U=\vect{u}{\bar{u}} \,, \;\; V=\vect{v}{\bar{v}}\,.
\end{equation}
We denote by $\mathfrak{F^*}$ its adjoint with 
respect to the scalar product $\eqref{comsca}$
\begin{equation*}
(\mathfrak{F}U,V)_{L^{2}}=(U,\mathfrak{F}^{*}V)_{L^2}\,, 
\quad \forall\,\, U,\, V\in 
L^{2}(\mathbb{T}^{d};\mathbb{C}^{2})\cap \mathcal{U}\,,\qquad 
\mathfrak{F}^*:=\left(\begin{matrix} A^* & \ov{B}^* \\ {B}^* & \ov{A}^*\end{matrix}\right)\,,
\end{equation*}
where $A^*$ and $B^*$ are respectively 
the adjoints of the operators $A$ and $B$ with respect to
the complex scalar product on 
$L^{2}(\mathbb{T}^{d};\mathbb{C})$
in \eqref{scalarL}.
\begin{definition}{\bf (Self-adjointness).}\label{def:selfself}
An  operator $\mathfrak{F}$ of the 
form \eqref{barrato4} is \emph{self-adjoint}
if and only if
\begin{equation}\label{selfFFFF}
 A^{*}=A,\;\;
\;\; \ov{B}=B^{*}\,.
\end{equation}
\end{definition}


\begin{remark}{\bf (Matrices of symbols).}\label{simboMatr}
Recall Remark \ref{strutAlg}. Consider two symbols
$a_1,a_2\in \mathcal{N}_s^{m}$ and the matrix
\begin{equation}\label{vinello100}
A := A(x,\x):=
\left(\begin{matrix} a_{1}(x,\x) & a_{2}(x,\x) \vspace{0.2em}\\
\ov{a_{2}(x,-\x)} & \ov{a_{1}(x,-\x)}
\end{matrix}
\right)\,.
\end{equation}
Define the operator (recall  \eqref{notaPara})
\[
M:=\opbw(A(x,\x)):=
\left(\begin{matrix} \opbw(a_{1}(x,\x)) & \opbw(a_{2}(x,\x)) \vspace{0.2em}\\
\opbw(\ov{a_{2}(x,-\x)}) & \opbw(\ov{a_{1}(x,-\x)})
\end{matrix}
\right)\,.
\]
Recalling \eqref{simboBarrato}, 
\eqref{simboAggiunto}, 
one can note that $M$ is \emph{real-to-real}. Moreover
$M$ is \emph{self-adjoint} if and only if
\begin{equation}\label{simboAggiunto2}
a_{1}(x,\x)=\ov{a_1(x,\x)}\,,
\qquad
{a_{2}(x,-\x)}={a_2(x,\x)}\,.
\end{equation}
\end{remark}

\subsection{Non-homogeneous symbols}\label{sec:nonomosimboli}
In this section we study some properties of
symbols depending nonlinearly on some function 
$u\in H^{s}(\mathbb{T}^{d};\mathbb{C})$.
We recall classical tame estimates for composition of functions
(see for instance \cite{Moser-Pisa-66}, \cite{Rabinowitz}, \cite{Tay-Para}).
A function $f: \mathbb{T}^{d}\times B_R\to \mathbb{C}$, where $B_R:=\{y\in \mathbb{R}^{m} : |y|<R\}$, $R>0$, 
induces the composition operator (Nemytskii)
\begin{equation}\label{compoNEM}
\tilde{f}(u):=f(x,u(x),Du(x),\ldots, D^{p}u(x))\,,
\end{equation}
where $D^{k}u(x)$ denote the derivatives $\pa_{x}^{\alpha}$ of order $|\alpha|=k$
(the number $m$ of $y$-variables  depends on $p$, $d$).

\begin{lemma}{\bf (Lipschitz estimates).}\label{CompfuncLemma}
Fix $\gamma>0$ 
and assume that 
$f\in C^{\infty}(\mathbb{T}^{d}\times B_R; \mathbb{R})$.
Then, for any $u\in H^{\gamma+p}$ with $\|u\|_{W^{p,\infty}}<R$,  one has
\begin{align}
& \|\tilde{f}(u)\|_{H^{\gamma}}\leq C \|f\|_{C^{\gamma}}(1+\|u\|_{H^{\gamma+p}})\,,\\
&\|\tilde{f}(u+h)-\tilde{f}(u)\|_{H^{\gamma}}\leq C \|f\|_{C^{\gamma+1}}
(\|h\|_{H^{\gamma+p}}
+\|h\|_{W^{p,\infty}}\|u\|_{H^{\gamma+p}})\,,\\
&\|\tilde{f}(u+h)-\tilde{f}(u)-(d_{u}\tilde{f})(u)[ h]\|_{H^{\gamma}}\leq
\nonumber\\
&\qquad \qquad \qquad \qquad
C \|f\|_{C^{\gamma+2}}\|h\|_{W^{p,\infty}}
(\|h\|_{H^{\gamma+p}}
+\|h\|_{W^{p,\infty}}\|u\|_{H^{\gamma+p}})\,,
\end{align}
for any $h\in H^{\gamma+p}$ with $\|h\|_{W^{p,\infty}}<R/2$ and where $C>0$ is a constant depending on $\gamma$ and the norm $\|u\|_{W^{p,\infty}}$.
\end{lemma}

Consider the function $F(u,\nabla u)$ introduced after formula \eqref{NLSnon}
%
%
and  a symbol $f(\x)$, independent of $x\in \mathbb{T}^{d}$, such that
$|f|_{\mathcal{N}_s^{m}}\leq C<+\infty$, for some 
constant $C$.
Let us define the symbol
\begin{equation}\label{nonomosimbo}
a(x,\x):= \big( \pa_{z_j^{\alpha} z_k^{\beta}}F\big)(u,\nabla u)
f(\x)\,, \quad z_j^{\alpha}:=\pa_{x_j}^{\alpha}u^{\s}, z_k^{\beta}:=\pa_{x_k}^{\beta}u^{\s'}  
\end{equation}
for some $j,k=1,\ldots,d$, $\alpha,\beta\in\{0,1\}$ and $\s,\s'\in \{\pm\}$
where we used the notation $u^{+}=u$ and $u^{-}=\bar{u}$.
We have the following.
\begin{lemma}\label{lem:nonomosimbo}
Fix $s_0>d/2$.
For $u\in B_R( H^{s+s_0+1}(\mathbb{T}^{d};\mathbb{C}))$, 
$s\in \mathbb{N}$,
we have
\begin{equation}\label{nonomosimbo2}
|a|_{\mathcal{N}_s^{m}}\lesssim C \|u\|_{H^{s+s_0+1}}\,,
\end{equation}
where $C>0$ is some constant depending on $\|u\|_{H^{s+s_0+1}}$
and bounded from above when $u$ goes to zero.
Moreover,
for any $h\in H^{s+s_0+1}$, the map 
$h\to (\pa_{u}a)(u;x,\x) h$ is a $\mathbb{C}$-linear map from
$ H^{s+s_0+1}$ to $\mathbb{C}$ and satisfies 
\begin{equation}\label{simboStima2}
|(\pa_{u}a) h|_{\mathcal{N}_{s}^{m}}\lesssim C\|h\|_{H^{s+s_0+1}}\,,
\end{equation}
for some constant $C>0$ as above. 
The same holds for $\pa_{\bar{u}}a$.
\end{lemma}
\begin{proof}
It follows by Lemma \ref{CompfuncLemma} applied 
on the function  
$\big( \pa_{z_j^{\alpha} z_k^{\beta}}F\big)(u,\nabla u)f(\x)$,
see \eqref{nonomosimbo}. 
\end{proof}

\section{Paralinearization of NLS}\label{sec:3}
We now paralinearize the Hamiltonian
 nonlinearity $P(u)$ in \eqref{NLSnon}. 
\begin{lemma}\label{product}
Fix $s_0>d/2$ and $0\leq\rho<s-s_0$, $s\geq s_0$.
Consider $u\in H^{s}(\mathbb{T}^{d};\mathbb{C})$.
Then we have that 
\begin{align}
P(u)&=T_{\pa_{u\bar{u}}F}[u]+T_{\pa_{\bar{u}\,\bar{u}}F}[\bar{u}]
\label{paralin1}
\\&+\sum_{j=1}^{d}\Big(  T_{\pa_{\bar{u}u_{x_{j}}}F}[u_{x_j}]
+T_{\pa_{\bar{u}\,\ov{u_{x_{j}}}}F}[\ov{u_{x_j}}] \Big)
-\sum_{j=1}^{d}\pa_{x_j}\Big(  T_{\pa_{{u}\ov{u_{x_{j}}}}F}[u]
+T_{\pa_{\bar{u}\,\ov{u_{x_{j}}}}F}[\bar{u}] \Big)\label{paralin2}\\
&-\sum_{j=1}^{d}\pa_{x_{j}}
\sum_{k=1}^{d}\Big( 
 T_{\pa_{\ov{u_{x_{j}}} \,{u_{x_k}}}F}[u_{x_k}]
+T_{\pa_{\ov{u_{x_{j}}}\, \ov{u_{x_{k}}}}F}[\bar{u_{x_{k}}}]
\Big)+R(u)\,,\label{paralin3}
\end{align}
where  $R(u)$ is a remainder satisfying 
\begin{equation}\label{stimarestopara}
\|R(u)\|_{H^{s+\rho}}\lesssim C\|u\|_{H^{s}}^{2}\,,
\end{equation}
for some constant $C>0$ depending on  $\|u\|_{H^{s}}$ bounded as $u$ goes to $0$.
\end{lemma}
\begin{proof}
The \eqref{paralin1}-\eqref{paralin3} 
follow by the Bony para-linearization formula, 
see Lemma \ref{lem:paraproduct}
(see also \cite{Metivier}, \cite{Tay-Para}).
\end{proof}
We now rewrite the equation \eqref{QLNLS}
as a para-differential system.
Let us introduce the symbols
\begin{equation}\label{simboa2}
\begin{aligned}
&a_2(x,\x):=a_{2}(U;x,\x)
:=\sum_{j,k=1}^{d}(\pa_{\bar{u_{x_k}} u_{x_{j}}}F) \x_{j}\x_{k}\,,
\\&
b_2(x,\x):=b_{2}(U;x,\x)
:=\sum_{j,k=1}^{d}(\pa_{\bar{u_{x_k}}\, \bar{u_{x_{j}}}}F) \x_{j}\x_{k}\,,\\
&a_1(x,\x):=a_{1}(U;x,\x):=\frac{\ii}{2}\sum_{j=1}^{d}\Big(
(\pa_{\ov{u} u_{x_{j}}}F)-(\pa_{{u} \ov{u_{x_{j}}}}F)
\Big)\x_{j}\,,
\end{aligned}
\end{equation}
where $F=F(u,\nabla u)$ in \eqref{HamNLS}.

\begin{lemma}\label{realtaAAA}
One has that
\begin{equation}\label{realtaAAA1}
a_2(x,\x)=\ov{a_2(x,\x)}\,,\quad
a_1(x,\x)=\ov{a_1(x,\x)}\,, \quad a_1(x,-\x)=-a_1(x,\x)\,,
\quad a_2(x,-\x)=a_2(x,\x)\,,
\end{equation}
\begin{equation}\label{realtaAAA2}
\begin{aligned}
|a_{2}|_{\mathcal{N}_p^{2}}+
|b_{2}|_{\mathcal{N}_p^{2}}
+|a_{1}|_{\mathcal{N}_p^{1}}&\lesssim C\|u\|_{H^{p+s_0+1}}\,,
\quad \forall \, p+s_0\leq s\,,\,\quad p\in \mathbb{N}\,,
\end{aligned}
\end{equation}
for some constant $C>0$ depending on  $\|u\|_{H^{p+s_0+1}}$ bounded as $u$ goes to $0$.
\end{lemma}
\begin{proof}
The \eqref{realtaAAA1} follows by direct inspection using \eqref{simboa2}.
The \eqref{realtaAAA2} follows by Lemma \ref{lem:nonomosimbo}.
\end{proof}

The following holds true.

\begin{proposition}{\bf (Paralinearization of NLS).}\label{NLSparapara}
We have that
the equation \eqref{QLNLS}
is equivalent to the following system (recall  \eqref{notaPara}):
\begin{equation}\label{QLNLS444}
\dot{U}=\ii E\opbw\big(|\x|^{2}\uno +A_{2}(x,\x)+A_{1}(x,\x)\big)U+
R(U)U\,,
\end{equation}
where 
\begin{equation*}
U:=\vect{u}{\bar{u}}\,,\quad  E:=\sm{1}{0}{0}{-1}\,,\quad
\uno:=\sm{1}{0}{0}{1}\,,
\end{equation*}
the matrices $A_2(x,\x)=A_2(U;x,\x)$,
$A_1(x,\x)=A_1(U;x,\x)$  have the form
\begin{equation}\label{matriceA2}
A_2(x,\x):=\left(\begin{matrix}a_2(x,\x) & b_{2}(x,\x) \vspace{0.2em}\\
\ov{b_{2}(x,-\x)} & {a_{2}(x,\x)} \end{matrix}\right)\,,
\qquad
A_1(x,\x):=\left(\begin{matrix}a_1(x,\x) & 0 \vspace{0.2em}\\
0 & \ov{a_{1}(x,-\x)} \end{matrix}\right)
\end{equation}
and  $a_2,a_1,b_2$ are  the symbol in \eqref{simboa2}.
The operators $\opbw(A_{i}(x,\x))$ are self-adjoint (see \eqref{selfFFFF}).
The remainder $R(U)$ is a $2\times2$ matrix of operators (see \eqref{vinello2})
which is \emph{real-to-real}, i.e. satisfies \eqref{vinello}.
Moreover, for any $s\geq2(d+1)+3$ and any $ U,V\in H^{s}(\mathbb{T}^{d};\mathbb{C}^{2}) $, 
it satisfies the estimates 
\begin{align}
\|R(U)U\|_{H^{s}}&\lesssim C\|U\|_{H^{s}}^{2}\,,\qquad C:=C(\|U\|_{H^{s}})\,,\label{stimaRRR}\\
\|R(U)[U]-R(V)[V]\|_{H^{s}}&\lesssim
 C_2 (\| U\|_{H^{s}}+\| V\|_{H^{s}})\|U-V\|_{H^{s}}\,, \;\; C_2:=C_2(\|U\|_{H^{s}},\|V\|_{H^{s}})\,,\label{nave101}
 \end{align}
for some constants $C,C_2>0$  bounded as $u$ and $v$ go to $0$.
\end{proposition}

\begin{proof}
We start by noting that
\begin{equation}\label{facile}
\pa_{x_{j}}=\opbw(\ii\x_{j})\,,\;\; j=1,\ldots \,d\,,
\end{equation}
and that the quantization of  a symbol $a(x)$ is given by $\opbw(a(x))$.
We also remark that
the symbols 
appearing in \eqref{paralin1}, \eqref{paralin2} and \eqref{paralin3}
can be estimated  (in the norm $|\cdot|_{\mathcal{N}_{s}^{m}}$)
by using Lemma \ref{realtaAAA}.
Consider now the first para-differential term in \eqref{paralin3}. 
We have, for any $j,k=1,\ldots,d$,
\[
\pa_{x_{j}}
 T_{\pa_{\ov{u_{x_{j}}} \,{u_{x_k}}}F}\pa_{x_k}u=
 \opbw(\ii\x_{j})\circ\opbw(\pa_{\ov{u_{x_{j}}} \,{u_{x_k}}}F)\circ\opbw(\ii \x_k)u\,.
\]
By applying 
Proposition \ref{prop:compo} 
and recalling the 
Poisson  bracket in \eqref{PoissonBra},
we deduce 
\begin{align}
 \opbw(\ii\x_{j})\,\circ&\,\opbw(\pa_{\ov{u_{x_{j}}} \,{u_{x_k}}}F)\circ\opbw(\ii \x_k)
=
\opbw\big( -\x_{j}\x_{k}\pa_{\ov{u_{x_{j}}} \,{u_{x_k}}}F \big)
\\&
 +\opbw\Big(
  \frac{\ii}{2 } \x_{k}\pa_{x_{j}}(\pa_{\ov{u_{x_{j}}} \,{u_{x_k}}}F)
 -\frac{\ii\x_{j}}{2}\pa_{x_{k}}(\pa_{\ov{u_{x_{j}}} \,{u_{x_k}}}F)\Big)
    \\&+\widetilde{R}^{(1)}_{j,k}(u) +\widetilde{R}^{(2)}_{j,k}(u)\,,
\end{align}
where $\widetilde{R}^{(1)}_{j,k}(u):=
\opbw\big( 
 -\frac{1}{4}\pa_{{x_{j}}{x_{k}}}(\pa_{\ov{u_{x_{j}}} \,{u_{x_k}}}F)\big)$
 and $\widetilde{R}^{(2)}_{j,k}(u)$ is some bounded operator.
 More precisely, using \eqref{composit2}, \eqref{actionSob} 
 and the estimates given by Lemma \ref{lem:nonomosimbo},
 we have, $\forall \, h\in H^{s}(\mathbb{T}^{d};\mathbb{C})$,
\begin{equation}\label{crociate}
\|\widetilde{R}^{(2)}_{j,k}(u)h\|_{H^{s}}\lesssim C\|h\|_{H^{s}}\|u\|_{H^{s}}\,,\qquad 
\|\widetilde{R}^{(1)}_{j,k}(u)h\|_{H^{s}}\lesssim C\|h\|_{H^{s}}\|u\|_{H^{2s_0+3}}\,,
\end{equation}
for some constant $C>0$ depending on  
$\|u\|_{H^{s}}$ bounded as $u$ goes to $0$, with $s_0\geq d+1$, 
$s_0\in\mathbb{N}$.
We set
\[
\widetilde{R}(u):=\sum_{j,k=1}^{d}\Big(\widetilde{R}^{(1)}_{j,k}(u)
+\widetilde{R}^{(2)}_{j,k}(u)\Big)\,.
\]
Then 
\[
\begin{aligned}
-\sum_{j,k=1}^{d}\pa_{x_j} &T_{\pa_{\ov{u_{x_{j}}} \,{u_{x_k}}}F}\pa_{x_{k}}u=
\opbw\Big(\sum_{j,k=1}^{d}\x_{j}\x_{k}\pa_{\ov{u_{x_{j}}} \,{u_{x_k}}}F\Big)+
\widetilde{R}(u)
\\&
-\frac{\ii}{2}\opbw\Big(
\sum_{j,k=1}^{d}\Big( -\x_{j}\pa_{x_{k}}(\pa_{\ov{u_{x_{j}}} \,{u_{x_k}}}F)
+ \x_{k}\pa_{x_{j}}(\pa_{\ov{u_{x_{j}}} \,{u_{x_k}}}F)
\Big)\Big)
\\&
\stackrel{\eqref{simboa2}}{=}
\opbw(a_2(x,\x))+\widetilde{R}(u)
+\frac{\ii}{2}\opbw\Big(\sum_{j,k=1}^{d} \x_{j}\pa_{x_{k}}\Big(
(\pa_{\ov{u_{x_{j}}} \,{u_{x_k}}}F)-(\pa_{\ov{u_{x_{k}}} \,{u_{x_j}}}F)\Big)
\Big)\\
&
=\opbw(a_2(x,\x))+\widetilde{R}(u)\,,
\end{aligned}
\]
where we used the symmetry of the matrix $\pa_{\ov{\nabla u}\, \nabla u}F$ (recall $F$ is real).
By performing similar explicit computations on the other summands in \eqref{paralin1}-\eqref{paralin3}
we get the 
\eqref{QLNLS444}, \eqref{matriceA2} with symbols in 
\eqref{simboa2}.
By the discussion above we deduced that the remainder $R(U)$ in \eqref{QLNLS444}
satisfies the bound \eqref{stimaRRR}.
The estimate \eqref{nave101} can be deduced by reasoning as in Lemma $4.5$
in \cite{Feola-Iandoli-Loc}.
\end{proof}

\section{Basic energy estimates}\label{sec:4}
Fix $s_0\geq d+1$, $s_0\in\mathbb{N}$, $s\geq 2s_0+7$,
$T>0$, and consider a function $u$ such that
\begin{equation}\label{regularity}
\begin{aligned}
&u\in L^{\infty}([0,T); H^{s}(\mathbb{T}^{d}; \mathbb{C}))\cap
Lip([0,T); H^{s-2}(\mathbb{T}^{d}; \mathbb{C}))\,,
\\
& \sup_{t\in[0,T)}\|u(t)\|_{H^{2s_0+7}}+
\sup_{t\in[0,T)}\|\pa_{t}u(t)\|_{H^{2s_0+5}}
\leq \mathtt{r}\,,
\end{aligned}
\end{equation}
for some $\mathtt{r}>0$.
Let $U:=\vect{u}{\bar{u}}\in \mathcal{U}$  (recall \eqref{spazioUU}).
Consider the system
\begin{equation}\label{QLNLSV}
\left\{\begin{aligned}
&\dot{V}=\ii E\opbw\big(|\x|^{2}\uno +A_{2}(x,\x)+A_{1}(x,\x)\big)V\,,\\
&V(0)=V_0:=U(0)\,,
\end{aligned}\right.
\end{equation}
where $A_i$, $i=1,2$, are the matrices of symbols given by Proposition \ref{NLSparapara}.
We shall provide \emph{a priori} energy estimates for the equation \eqref{QLNLSV}.

\begin{theorem}{\bf (Energy estimates).}\label{StimeEnergia}
Assume \eqref{regularity}. 
Then for $s\geq 2s_0+7$ the following holds. 
If a function $V=\vect{v}{\bar{v}}\in \mathcal{U}$, with
$v\in L^{\infty}([0,T); H^{s}(\mathbb{T}^{d}; \mathbb{C}))\cap
Lip([0,T); H^{s-2}(\mathbb{T}^{d}; \mathbb{C}))$
solves the problem \eqref{QLNLSV}, with initial condition $v(0)\in H^{s}(\mathbb{T}^{d};\mathbb{C})$, 
then one has
\begin{equation}\label{StimeEnergia2}
\|v(t)\|_{H^{s}}^{2}\lesssim_{\mathtt{r}} \|v(0)\|_{H^{s}}^{2}+\int_{0}^{t}C\|u(\s)\|_{H^{s}}
\|v(\s)\|^{2}_{H^{s}}d\s\,,
\quad \mbox{for almost every}\, t\in[0,T)\,,
\end{equation}
for some $C>0$ depending on 
$\|u\|_{H^{s}}$, 
and bounded from above as
$\|u\|_{H^{s}}$ goes to zero.
\end{theorem}
The proof of the Theorem  above requires some preliminary results
which will be proved in the following subsections.

\subsection{Block-diagonalization}
The aim of this section is to block-diagonalize system \eqref{QLNLSV}
up to bounded remainders.
This will be achieved into two steps.
In the following, for simplicity, 
sometimes 
we  omit the dependence on $(x,\x)$ from the symbols.

\subsubsection{Block-diagonalization at highest order}
Consider the matrix of symbols
\begin{equation}\label{A2tilde}
\begin{aligned}
&E(\uno+\widetilde{A}_{2,\mathtt{R}}(x,\x))\,,\qquad 
\widetilde{A}_{2,\mathtt{R}}(x,\x)
:=
\left(\begin{matrix}\widetilde{a}_{2,\mathtt{R}}(x,\x) & 
\widetilde{b}_{2,\mathtt{R}}(x,\x) \vspace{0.2em}\\
\ov{\widetilde{b}_{2,\mathtt{R}}(x,-\x)} & 
{\widetilde{a}_{2,\mathtt{R}}(x,\x)} \end{matrix}\right)\,,\\
&\widetilde{a}_{2,\mathtt{R}}(x,\x):=|\x|^{-2}a_2(x,\x)\mathcal{X}_{\mathtt{R}}(\x)\,,\quad
\widetilde{b}_{2,\mathtt{R}}(x,\x):=|\x|^{-2}b_2(x,\x)\mathcal{X}_{\mathtt{R}}(\x)\,,
\end{aligned}
\end{equation}
where $a_2(x,\xi)$ and $b_2(x,\xi)$ are defined in \eqref{simboa2}
and $\mathcal{X}_{\mathtt{R}}(\x)$ is the cut-off function in  \eqref{cutofftheta}.
Note that the symbols above are well-defined thanks to the cut-off function $\mathcal{X}_{\mathtt{R}}(\x)$.

Define 
\begin{equation}\label{nuovadiag}
\begin{aligned}
\lambda_{2}(x,\x)&:= \sqrt{(1+\widetilde{a}_{2,\mathtt{R}}(x,\x))^{2}
-|\widetilde{b}_{2,\mathtt{R}}(x,\x)|^{2}},
\qquad \widetilde{a}_{2, \mathtt{R}}^{+}(x,\x):=\lambda_{2}(x,\x)-1.
\end{aligned}
\end{equation}
Notice that the symbol $\lambda_{2}$ is well-defined by Hypothesis \ref{hyp1}.
The matrix of the normalized eigenvectors associated to the eigenvalues of 
$E(\uno+\widetilde{A}_{2,\mathtt{R}}(x,\x))$
is 
\begin{equation}\label{transC}
\begin{aligned}
S(x,\x)&:=\left(\begin{matrix} {s}_1(x,\x) & {s}_2(x,\x)\vspace{0.2em}\\
{\ov{s_2(x,\x)}} & {{s_1(x,\x)}}
\end{matrix}
\right)\,,
\qquad
S^{-1}(x,\x):=\left(\begin{matrix} {s}_1(x,\x) & -{s}_2(x,\x)\vspace{0.2em}\\
-{\ov{s_2(x,\x)}} & {{s_1(x,\x)}}
\end{matrix}
\right)\,,
\\
s_{1}&:=\frac{1+\widetilde{a}_{2,\mathtt{R}}
+\lambda_{2}}{\sqrt{2\lambda_{2}\big(1+\widetilde{a}_{2,\mathtt{R}}+
\lambda_{2}\big) }},
\qquad s_{2}:=\frac{-\widetilde{b}_{2,\mathtt{R}}}{\sqrt{2\lambda_{2}\big(1+\widetilde{a}_{2,\mathtt{R}}
+\lambda_{2}\big) }}\,.
\end{aligned}
\end{equation}
Let us also define  (recall \eqref{PoissonBra}, \eqref{PoissonBra22})
\begin{equation}\label{simboliAA11}
\begin{aligned}
S_1(x,\x)&:=
\left(
\begin{matrix}
s_{1}^{(1)}(x,\x) & s_{2}^{(1)}(x,\x) \vspace{0.2em}\\ 
\ov{s_{2}^{(1)}(x,-\x) } & 
{s_{1}^{(1)}(x,-\x)}
\end{matrix}
\right)\\
&:=\frac{1}{2\ii}\left( 
\begin{matrix}  \{s_{2},\overline{s_{2}}\}(x,\xi) & 2\{s_{1},{s_{2}}\}(x,\xi)\vspace{0.2em}\\
-{2\{s_{1},\overline{s_{2}}\}(x,-\xi)} & {\{s_{2},\overline{s_{2}}\}(x,-\xi)}
\end{matrix}
\right)S(x,\x)\,,
\end{aligned}
\end{equation}
and 
\begin{equation}\label{simboliAA12}
\begin{aligned}
&g_{1}(x,\x)
:=\frac{1}{2\ii}\{s_{1}^{(1)},{s_{1}}\}
-\frac{1}{2\ii}\{s_{2}^{(1)},\ov{s_{2}}\}
+\frac{1}{8}\s(s_{2},\ov{s_{2}})\,,
\\
&g_{2}(x,\x)
:=-\frac{1}{2\ii}\{s_{1}^{(1)},{s_{2}}\}
+\frac{1}{2\ii}\{s_{2}^{(1)},{s_{1}}\}\,,\\
&S_2(x,\x):=
\left(
\begin{matrix}
s_{1}^{(2)}(x,\x) & s_{2}^{(2)}(x,\x) \vspace{0.2em}\\ 
\ov{s_{2}^{(2)}(x,-\x) } & 
{s_{1}^{(2)}(x,-\x)}
\end{matrix}
\right)
:=-\left( 
\begin{matrix} g_{1}(x,\x) & g_{2}(x,\x)\vspace{0.2em}\\
{\ov{g_{2}(x,-\x)}} & {{g_{1}(x,-\x)}}
\end{matrix}
\right)S(x,\x)\,.
\end{aligned}
\end{equation}
We have the following lemma.
\begin{lemma}\label{StimeMatSSS}
We have that
the symbols $\widetilde{a}^{+}_{2,\mathtt{R}}$ in \eqref{nuovadiag}, 
$\widetilde{a}_{2,\mathtt{R}}$, 
$\widetilde{b}_{2,\mathtt{R}}$ in \eqref{A2tilde}, 
$s_1,s_2$ in \eqref{transC},
$g_{1}, g_{2}$ in \eqref{simboliAA12}
are even in the variable $\x\in \mathbb{R}^{d}$, while the
symbols in the matrix \eqref{simboliAA11} 
are odd in $\x\in \mathbb{R}^{d}$. Let $s_0\geq d+1$, $p\in \mathbb{N}$. 
One has 
\begin{align}
&|\widetilde{a}^{+}_{2,\mathtt{R}}|_{\mathcal{N}^{0}_{p}}
+|\widetilde{a}_{2,\mathtt{R}}|_{\mathcal{N}^{0}_{p}}
+|\widetilde{b}_{2,\mathtt{R}}|_{\mathcal{N}^{0}_{p}}
+|s_1-1|_{\mathcal{N}^{0}_{p}}
+|s_2|_{\mathcal{N}^{0}_{p}}\lesssim
C_1\|u\|_{H^{p+s_0+1}}\,,\quad p+s_0+1\leq s\,,\label{achille55}\\
&| \{s_{2},{s_{1}}\}|_{\mathcal{N}^{-1}_{p}}+
| \{s_{2},\overline{s_{2}}\}|_{\mathcal{N}^{-1}_{p}}
+|g_{i}|_{\mathcal{N}^{-2}_{p}}\lesssim
C_2\|u\|_{H^{p+s_0+3}}\,,\;\;\;
p+s_0+3\leq s\,,\label{achille552}
\end{align}
for $i=1,2$, and
for some $C_1$ depending on  $\|u\|_{H^{p+s_0+1}}$
and $C_2$ depending on  $\|u\|_{H^{p+s_0+3}}$, both bounded as $u$ goes to zero.
\end{lemma}
\begin{proof}
The symbols are even in $\x$ by direct inspection 
using \eqref{A2tilde}, \eqref{transC} and \eqref{realtaAAA1}.
The symbols in \eqref{simboliAA11} are odd in $\x$ by the same
reasoning. 
Estimates \eqref{achille55}, \eqref{achille552}
follow by Lemma \ref{realtaAAA} since the symbols $s_1,s_2$ are regular functions 
of $\widetilde{a}_{2,\mathtt{R}}$, $\widetilde{b}_{2,\mathtt{R}}$ 
(recall also the \eqref{prodSimboli}).
\end{proof}

By a direct computation one can check that
\begin{equation}\label{diagodiago}
S^{-1}(x,\x)E(\uno+\widetilde{A}_{2,\mathtt{R}}(x,\x))S(x,\x)=
\sm{\lambda_2(x,\x)}{0}{0}{-\lambda_2(x,\x)}\,,
\qquad s_1^{2}-|s_2|^{2}=1\,.
\end{equation}
Moreover 
the matrices of symbols $S, S^{-1}$ in \eqref{transC},
$S_1$ in \eqref{simboliAA11} and $S_2$ in \eqref{simboliAA12}
have the form \eqref{vinello100}, i.e. they are
\emph{real-to-real}. 
We shall study how the system \eqref{QLNLS444}
transforms under the maps
\begin{equation}\label{Mappa1}
\begin{aligned}
\Phi&=\Phi(u)[\cdot]:=\opbw(S^{-1}(x,\x))\,,
\\
\Psi&=\Psi(u)[\cdot]:=\opbw(S(x,\x)+S_1(x,\x)+S_{2}(x,\x))\,.
\end{aligned}
\end{equation}

\begin{lemma}\label{propMAPPA}
Assume the \eqref{regularity}.
For any $s\in \mathbb{R}$ 
the following holds:
\begin{enumerate}
\item[$(i)$] there exists a constant $C$ depending on $s$ 
and  on $\|u\|_{H^{2s_0+3}}$,  bounded as $u$ goes to zero, 
such that
\begin{equation}\label{stime-descentTOTA}
\begin{aligned}
\|\Phi(u)V\|_{H^{s}}+\|\Psi(u)V\|_{H^{s}}&\leq 
\|{V}\|_{H^{s}}\big(1+C\|u\|_{H^{2s_0+3}}\big)\,,
\qquad \forall\, V\in H^{s}(\mathbb{T}^{d};\mathbb{C})\,;
\end{aligned}
\end{equation} 

\item[$(ii)$] one has $\Psi(u)[\Phi(u)[\cdot]]=\uno+Q(u)[\cdot]$
where $Q$ is a real-to-real remainder of the form \eqref{vinello2}
satisfying 
\begin{equation}\label{achille10}
\|Q(u)V\|_{H^{s+3}}\lesssim C \|V\|_{H^{s}}\|u\|_{H^{2s_0+7}}\,,
\end{equation}
\begin{equation}\label{achille10tris}
\|Q(u)V\|_{H^{s+2}}\lesssim C \mathtt{R}^{-1}\|V\|_{H^{s}}\|u\|_{H^{2s_0+7}}\,,
\end{equation}
for some $C>0$ depending on $\|u\|_{H^{2s_0+7}}$ and bounded 
as $u$ goes to zero;

\item[$(iii)$] for $\mathtt{R}>0$ large enough 
with respect to $\mathtt{r}>0$ in \eqref{regularity}
the map $\uno+Q(u)$ is invertible
and $(\uno+Q(u))^{-1}=\uno+\widetilde{Q}(u)$ with 
\begin{equation}\label{carmen2}
\|\widetilde{Q}(u)V\|_{H^{s+2}}\lesssim C \mathtt{R}^{-1}\|V\|_{H^{s}}\|u\|_{H^{2s_0+7}}\,,
\end{equation}
for some $C>0$ as in item $(ii)$. Moreover
$\Phi^{-1}(u):=(\uno+\widetilde{Q}(u))\Psi(u)$ satisfies
\begin{equation}\label{carmen}
\|\Phi^{-1}(u)V\|_{H^{s}}\leq 
\|{V}\|_{H^{s}}\big(1+C\|u\|_{H^{2s_0+7}}\big)\,,
\qquad \forall\, V\in H^{s}(\mathbb{T}^{d};\mathbb{C})\,,
\end{equation}
for some $C>0$ depending on $\|u\|_{H^{2s_0+7}}$ and bounded 
as $u$ goes to zero;

\smallskip
\item[$(iv)$] for almost any $t\in [0,T)$,
one has 
$\pa_{t}\Phi(u)[\cdot]=\opbw(\pa_{t}S^{-1}(x,\x))$
and
\begin{equation}\label{achille11}
|\pa_{t}S^{-1}(x,\x)|_{\mathcal{N}_{s_0}^{0}}\lesssim_{\mathtt{r}} 
C \|u\|_{H^{2s_0+3}}\,,
\qquad
\|\pa_{t}\Phi(u)V\|_{H^{s}}\lesssim_{\mathtt{r}} C \|V\|_{H^{s}}\|u\|_{H^{2s_0+3}}\,,
\end{equation}
for some $C>0$ depending on $\|u\|_{H^{2s_0+3}}$ bounded 
as $u$ goes to zero.
\end{enumerate}
\end{lemma}

\begin{proof}
$(i)$ The bound \eqref{stime-descentTOTA} follows by 
\eqref{actionSob} and \eqref{achille55}, \eqref{achille552}.

\noindent
$(ii)$ By applying Proposition \ref{prop:compo} to the maps in  \eqref{Mappa1},
using the expansion \eqref{composit}
  and the \eqref{simboliAA11}, \eqref{simboliAA12}
  we have $\Psi(u)[\Phi(u)[\cdot]]=\uno+Q(u)$. The remainder $Q(u)$
  satisfies \eqref{achille10}, \eqref{achille10tris}
  by estimates 
  \eqref{composit2}, \eqref{composit2Theta} and 
     \eqref{achille55}, \eqref{achille552}.

\noindent
$(iii)$
This item follows by using Neumann series, the second condition in \eqref{regularity}, the bound
\eqref{achille10tris} and taking $\mathtt{R}$ 
large enough to obtain the smallness condition
$\|Q(u)V\|_{H^{s+2}}\lesssim 1/2 \|V\|_{H^{s}}$.
The \eqref{carmen} follows by composition using 
\eqref{carmen2} and \eqref{stime-descentTOTA}.

\noindent
$(iv)$ We note that
 \[
 \pa_{t}s_1(x,\x)=(\pa_{u}s_1)(u;x,\x)[\dot{u}]+
 (\pa_{\bar{u}}s_1)(u;x,\x)[\dot{\bar{u}}]\,.
 \]
By hypothesis \eqref{regularity} we have that $\dot{u}$ and $\dot{\bar{u}}$
belong to $H^{s-2}(\mathbb{T}^{d};\mathbb{C})$.
Moreover, recalling \eqref{transC} and \eqref{A2tilde}, 
we can express $\pa_{t}s_1(x,\x)$ in terms of the derivatives of the symbols
$a_2(x,\x), b_2(x,\x)$ in \eqref{simboa2}. 
Therefore, by applying 
Lemma \ref{lem:nonomosimbo} (see estimate \eqref{simboStima2}),
 we deduce
\[
|\pa_{t}s_1(x,\x)|_{\mathcal{N}_{s_0}^{0}}\lesssim_{\mathtt{r}} C\|u\|_{H^{2s_0+3}}\,.
\]
Reasoning similarly one can prove a similar bound for
the symbol $s_2$.
This implies 
 the first in \eqref{achille11}. 
 The second one follows by \eqref{actionSob}.
\end{proof}
We are ready to prove the following conjugation result.
\begin{proposition}{\bf (Block-diagonalization).}\label{prop:block}
Assume \eqref{regularity}, consider the system \eqref{QLNLSV}
and set
\begin{equation}\label{nuovavariabile}
Z:=\Phi(u)[V]\,.
\end{equation}
Then 
we have
\begin{equation}\label{systZZZ}
\dot{Z}=\ii E \opbw
\left(|\x|^{2}\uno+A_{2}^{(1)}(x,\x)+A_{1}^{(1)}(x,\x)
\right)Z+\mathcal{R}(U)V\,,
\end{equation}
where (recall \eqref{nuovadiag})
\begin{equation}\label{A1finale}
\begin{aligned}
&A_2^{(1)}(x,\x):=\left( \begin{matrix}
a_2^{(1)}(x,\x) & 0\\
0 &a_2^{(1)}(x,\x)
\end{matrix}\right)\,, \qquad a_2^{(1)}(x,\x):=|\x|^{2}\widetilde{a}_{2,\mathtt{R}}^{+}(x,\x)\,,\\
&A_1^{(1)}(x,\x):=\left( \begin{matrix}
a_1^{(1)}(x,\x) & b_1^{(1)}(x,\x) \vspace{0.2em}\\
\ov{b_1^{(1)}(x,-\x) } & a_1^{(1)}(x,-\x)
\end{matrix}\right)\,,\qquad a_i^{(1)}(x,\x)\in \mathbb{R}\,,
\; i=1,2\,,
\\&
a_1^{(1)}(x,-\x)=-a_1^{(1)}(x,\x)\,,
\qquad
b_1^{(1)}(x,-\x)=-b_1^{(1)}(x,\x)
\qquad
a_2^{(1)}(x,-\x)=a_2^{(1)}(x,\x) \,,
\end{aligned}
\end{equation}
and the symbols $a_2^{(1)}, a_1^{(1)}, b_1^{(1)}$ satisfy
\begin{align}
|a_{2}^{(1)}|_{\mathcal{N}_{p}^{2}}
&\lesssim \mathtt{c}_1 \|u\|_{H^{p+s_0+1}}\,,
\qquad p+s_0+1\leq s\,,\;\;\;p\in \mathbb{N}\label{A1finale2}\\
|a_1^{(1)}|_{\mathcal{N}_{p}^{1}}+ |b_1^{(1)}|_{\mathcal{N}_{p}^{1}}&\lesssim
\mathtt{c}_2 \|u\|_{H^{p+s_0+3}}\,,
\quad p+s_0+3\leq s\,,\;\;\;p\in \mathbb{N}\label{A1finale3}
\end{align}
for some $\mathtt{c}_1,\mathtt{c}_2>0$ depending respectively on $\|u\|_{H^{p+s_0+1}}$ and 
$\|u\|_{H^{p+s_0+3}}$, bounded as $u$ goes to zero.
The remainder $\mathcal{R}$ is real-to-real and satisfies,
for any $s\geq 2s_0+7$,
 the estimate
\begin{equation}\label{restofinale}
\|\mathcal{R}(U)V\|_{H^{s}}\lesssim_{\mathtt{r}} C\|V\|_{H^{s}}\|u\|_{H^{s}}\,,
\end{equation}
for some $C>0$ depending on $\|u\|_{H^{s}}$, 
bounded
as $u$ goes to zero.
\end{proposition}

\begin{proof}
By \eqref{QLNLSV}, \eqref{nuovavariabile}
we have
\begin{equation}\label{achille15}
\dot{Z}=\Phi(u)\ii E\opbw\big(|\x|^{2}\uno +A_{2}(x,\x)+A_{1}(x,\x)\big)V+(\pa_t\Phi(u))V\,.
\end{equation}
By item $(iii)$ of Lemma \ref{propMAPPA}
we can write $V=\Phi^{-1}(u)Z=(\uno+\widetilde{Q}(u))\Psi(u)Z$ with $\widetilde{Q}(u)$
 satisfying \eqref{carmen2}.
Using this  in \eqref{achille15} (recall \eqref{cutofftheta}) we get
\begin{equation}\label{achille16}
\dot{Z}=
\Phi(u)\ii E\opbw\Big(|\x|^{2}\uno +\mathcal{X}_{\mathtt{R}}(\x)
\big(A_{2}(x,\x)+A_{1}(x,\x)\big)\Big)\Psi(u)Z
+Q_1(u)V\,,
\end{equation}
where
\begin{equation}\label{achille18}
\begin{aligned}
Q_1(u)&:=\Phi(u)
\ii E\opbw\Big((1-\mathcal{X}_{\mathtt{R}}(\x))
\big(A_{2}(x,\x)+A_{1}(x,\x)\big)\Big)\Psi(u)\Phi(u)
+(\pa_t\Phi(u))\\
&+\Phi(u)
\ii E\opbw\Big(|\x|^{2}\uno +A_{2}(x,\x)+A_{1}(x,\x)\Big)
\widetilde{Q}(u)\Psi(u)\Phi(u)\,.
\end{aligned}
\end{equation}
By using \eqref{stime-descentTOTA}, 
\eqref{achille10}, \eqref{carmen2}, \eqref{actionSob}, \eqref{azionecufoffoff}, the \eqref{nonomosimbo2}
on the symbols $a_2,b_2,a_1$, and  item $(iv)$ of Lemma \ref{propMAPPA} 
we deduce
\begin{equation}\label{achille19}
\|Q_1(u)V\|_{H^{s}}\lesssim_{\mathtt{r}} C\mathtt{R}^{2}\|V\|_{H^{s}}\|u\|_{H^{s}}\,,
\qquad s\geq 2s_0+7\,,
\end{equation}
for some constant $C>0$ depending on $\|u\|_{H^{s}}$,
bounded
as $u$ goes to zero.
We now study the term of order one in \eqref{achille16}.
Recalling \eqref{Mappa1} we have
\begin{equation}\label{achille20}
\Phi(u)\ii E\opbw(\mathcal{X}_{\mathtt{R}}A_1(x,\x))\Psi(u)=
\opbw(S^{-1})
\ii E\opbw(\mathcal{X}_{\mathtt{R}} A_1)
\opbw(S)+Q_2(u)\,,
\end{equation}
where
\[
Q_{2}(u):=\Phi(u)\ii E\opbw(\mathcal{X}_{\mathtt{R}} A_1)\opbw(S_1+S_2)\,.
\]
Using Lemmata \ref{realtaAAA}, \ref{StimeMatSSS}, \ref{azioneSimboo} 
(recall that 
$S_1$, $S_2$ 
in \eqref{simboliAA11}, \eqref{simboliAA12}
are matrices of symbols of order $\leq -1$) 
one gets
\begin{equation}\label{timoria}
\|Q_2(u)V\|_{H^{s}}\lesssim C\|V\|_{H^{s}}\|u\|_{H^{2s_0+3}}\,,
\end{equation}
for some constant $C>0$ depending on $\|u\|_{H^{2s_0+3}}$,
bounded
as $u$ goes to zero.
We define
\begin{equation}\label{timoria6}
a_{i,\mathtt{R}}(x,\x):=\mathcal{X}_{\mathtt{R}}(\x)a_{i}(x,\x)\,,\;\;\; i=1,2\,,
\end{equation}
with $a_{2}(x,\x), a_{1}(x,\x)$ in \eqref{simboa2}.
Recalling  Lemma  \ref{realtaAAA},  \eqref{notaPara}
and \eqref{transC}
we have
\begin{equation}\label{achille21}
\begin{aligned}
&\opbw(S^{-1})
\opbw(\ii E \mathcal{X}_{\mathtt{R}}A_1)
\opbw(S)=\ii E
\left(
\begin{matrix}
C_1 & C_2\\
\ov{C_2} & \ov{C_1}
\end{matrix}
\right)
\\
&C_1:=T_{s_{1}}T_{a_{1,\mathtt{R}}}T_{s_{1}}
-T_{s_{2}}T_{a_{1,\mathtt{R}}}T_{\bar{s_{2}}}\,,\qquad 
C_2:=
T_{s_{1}}T_{a_{1,\mathtt{R}}}T_{s_{2}}
-T_{s_{2}}T_{a_{1,\mathtt{R}}}T_{{s_{1}}}\,.
\end{aligned}
\end{equation}
By using Proposition \ref{prop:compo} and the second condition in \eqref{diagodiago}
we get (see the expansion \eqref{composit})
\[
C_1=T_{a_{1,\mathtt{R}}}+Q_3(u)\,,\qquad C_2=Q_{4}(u)
\]
where $Q_{i}(u)$, $i=3,4$ are remainders satisfying, using  Lemmata 
\ref{realtaAAA}, \ref{StimeMatSSS},
 \begin{equation}\label{timoria2}
\|Q_i(u)V\|_{H^{s}}\lesssim C\|V\|_{H^{s}}\|u\|_{H^{2s_0+5}}\,,
\end{equation}
for some constant $C>0$ depending on $\|u\|_{H^{2s_0+5}}$,
bounded
as $u$ goes to zero.
%
Let us study the term of order two in \eqref{achille16}.
 By an explicit computation, using Proposition \ref{prop:compo} 
 and Lemma \ref{StimeMatSSS},
we have
\begin{equation}\label{achille20B}
\begin{aligned}
\Phi(u)\ii E\opbw\Big((|\x|^{2}\uno+\mathcal{X}_{\mathtt{R}}(\x)&A_2(x,\x))\Big)\Psi(u)=\ii E\left(
\begin{matrix}
B_1 & B_2\\
\ov{B_2} & \ov{B_1}
\end{matrix}
\right)\\&+\opbw\Big(
S^{-1}\ii E(|\x|^{2}\uno+\mathcal{X}_{\mathtt{R}}(\x)A_2)S_1)
\Big)+Q_{5}(u)
\end{aligned}
\end{equation}
where
 \begin{equation}\label{timoria3}
\|Q_5(u)V\|_{H^{s}}\lesssim C\|V\|_{H^{s}}\|u\|_{H^{2s_0+7}}\,,
\end{equation}
for some $C>0$ depending on $\|u\|_{H^{2s_0+7}}$,
bounded
as $u$ goes to zero,
and where, recalling Lemma  \ref{realtaAAA},  \eqref{notaPara} and \eqref{transC}, \eqref{timoria6},
\begin{equation}\label{achille21B}
\begin{aligned}
B_1&:=T_{s_{1}}T_{|\x|^{2}
+a_{2,\mathtt{R}}}T_{s_{1}}+T_{s_{1}}
T_{b_{2,\mathtt{R}}}T_{\bar{s_{2}}}
+T_{s_{2}}T_{\ov{b_{2,\mathtt{R}}}}T_{{s_{1}}}
+T_{s_{2}}
T_{|\x|^{2}+a_{{2,\mathtt{R}}}}
T_{\ov{s_{2}}}
\,,\\
B_2&:=
T_{s_{1}}T_{|\x|^{2}
+a_{2,\mathtt{R}}}T_{s_{2}}
+T_{s_{1}}T_{b_{2,\mathtt{R}}}T_{{s_{1}}}
+T_{s_{2}}T_{\ov{b_{2,\mathtt{R}}}}T_{{s_{2}}}
+T_{s_{2,\mathtt{R}}}T_{|\x|^{2}
+a_{2,\mathtt{R}}}T_{{s_{1}}}\,.
\end{aligned}
\end{equation}
We study each term separately.
By using Proposition \ref{prop:compo}
we get (see the expansion \eqref{composit})
\begin{equation}\label{achille30}
B_1:=T_{c_2}+T_{c_1}+Q_{6}(u)
\end{equation}
where
\begin{equation}\label{achille31}
\|Q_6(u)h\|_{H^{s}}\lesssim C\|h\|_{H^{s}}\|u\|_{H^{2s_0+5}}\,,
\end{equation}
for some constant $C>0$ depending on $\|u\|_{H^{2s_0+5}}$,
bounded
as $u$ goes to zero,
and 
\begin{equation}\label{achille32}
\begin{aligned}
c_2(x,\x)&:=(|\x|^{2}+a_{2,\mathtt{R}})(s_1^{2}+|s_2|^{2})
+b_{2,\mathtt{R}} s_1\ov{s_2}+\ov{b_{2,\mathtt{R}}}s_1s_2
\,,\\
c_1(x,\x)&:=\frac{1}{2\ii }\{s_{1},
(|\x|^{2}+a_{2,\mathtt{R}})s_{1}\}
+\frac{s_{1}}{2\ii }\{|\x|^{2}
+a_{2,\mathtt{R}}, s_{1}\}
\\&+\frac{1}{2\ii }\{s_{1},b_{2,\mathtt{R}}\ov{s_{2}}\}
+\frac{s_{1}}{2\ii }\{b_{2,\mathtt{R}},\ov{s_{2}}\}
+\frac{1}{2\ii }\{s_{2},\ov{b_{2,\mathtt{R}}}s_{1}\}
\\&+
\frac{s_{2}}{2\ii }\{\ov{b_{2,\mathtt{R}}},s_{1}\}
+\frac{1}{2\ii }\{s_{2},(|\x|^{2}
+a_{2,\mathtt{R}})\ov{s_{2}}\}+
\frac{s_{2}}{2\ii }\{|\x|^{2}
+a_{2,\mathtt{R}}, \ov{s_{2}}\}\,.
\end{aligned}
\end{equation}
By expanding the Poisson bracket (see \eqref{PoissonBra})
we get that
\begin{equation}\label{achille34}
c_2(x,\x)=\ov{c_2(x,\x)}\,, \qquad c_1(x,\x)=\ov{c_1(x,\x)}\,, 
\qquad
c_1(x,-\x)=-{c_1(x,\x)}\,\,.
\end{equation}
Moreover, by \eqref{achille55},  \eqref{prodSimboli}
and Lemma \ref{lem:nonomosimbo} on $a_2,b_2$,  we have
\begin{equation}\label{achille35}
|c_1|_{\mathcal{N}^{1}_{p}}\lesssim C\|u\|_{H^{p+s_0+2}}\,,
\end{equation}
for some $C$ depending on $\|u\|_{H^{p+s_0+2}}$,
bounded
as $u$ goes to zero.
Reasoning similarly we deduce
(see \eqref{achille21B})
\begin{equation}\label{achille30bis}
B_2:=T_{d_2}+T_{d_1}+Q_{7}(u)
\end{equation}
where
\begin{equation}\label{achille31bis}
\|Q_7(u)h\|_{H^{s}}\lesssim C\|h\|_{H^{s}}\|u\|_{H^{2s_0+5}}\,,
\end{equation}
for some $C$ depending on $\|u\|_{H^{2s_0+5}}$,
bounded
as $u$ goes to zero,
and 
\begin{equation}\label{achille32bis}
\begin{aligned}
d_2(x,\x)&:=(|\x|^{2}+a_{2,\mathtt{R}})s_1s_2
+b_2 s_1^{2}+\ov{b_{2,\mathtt{R}}}s_2^{2}\,,\\
d_1(x,\x)&:=\frac{1}{2\ii }\{s_{1},(|\x|^{2}
+a_{2,\mathtt{R}})s_{2}\}
+\frac{s_{1}}{2\ii }\{|\x|^{2}
+a_{2,\mathtt{R}}, s_{2}\}
\\&+\frac{1}{2\ii }\{s_{1},b_{2,\mathtt{R}} s_{1}\}
+\frac{s_{1}}{2\ii }\{b_{2,\mathtt{R}}, s_{1}\}
+\frac{1}{2\ii }\{s_{2},\ov{b_{2,\mathtt{R}}}s_{2}\}
\\&+
\frac{s_{2}}{2\ii }\{\ov{b_{2,\mathtt{R}}},s_{2}\}
+\frac{1}{2\ii }\{s_{2},(|\x|^{2}
+a_{2,\mathtt{R}})s_{1}\}+
\frac{s_{2}}{2\ii }\{|\x|^{2}
+a_{2,\mathtt{R}}, s_{1}\}\,.
\end{aligned}
\end{equation}
By expanding the Poisson bracket (see \eqref{PoissonBra})
we get that
\begin{equation}\label{achille34bis}
d_1(x,\x)\equiv0\,.
\end{equation}
We now study the second summand in the right hand side of \eqref{achille20B}
by computing explicitly the matrix of symbols of order 1.
Using \eqref{transC}, \eqref{simboliAA11}, \eqref{diagodiago}, 
we get
\[
S^{-1}
E(|\x|^{2}\uno+\mathcal{X}_{\mathtt{R}}(\x)A_2(x,\x))S_1
=
E\left(\begin{matrix}r_1(x,\x) & r_2(x,\x)
\vspace{0.2em}\\
\ov{r_2(x,-\x)} & \ov{r_1(x,-\x)}\end{matrix}
\right)\,,
\]
where
\begin{align}
&r_1(x,\x):=(|\x|^{2}+a_{2,\mathtt{R}})s_1s_1^{(1)}+
b_{2,\mathtt{R}}s_1\ov{s_2^{(1)}}
+\ov{b_{2,\mathtt{R}}}s_2s_1^{(1)}+
(|\x|^{2}+a_{2,\mathtt{R}})s_2\ov{s_2^{(1)}}\,,\label{cantolibero2}\\
&r_2(x,\x):=
(|\x|^{2}+a_{2,\mathtt{R}})s_1s_2^{(1)}+
b_{2,\mathtt{R}}s_1{s_1^{(1)}}
+\ov{b_{2,\mathtt{R}}}s_2s_2^{(1)}+
(|\x|^{2}+a_{2,\mathtt{R}})s_2{s_1^{(1)}}\,.\label{cantolibero3}
\end{align}
Moreover, using Lemma \ref{StimeMatSSS}, one can check that
 the symbols $r_1,r_2$  satisfy
\begin{equation}\label{achille3444}
 r_1(x,\x)=\ov{r_1(x,\x)}\,, 
\qquad
r_1(x,-\x)=-{r_1(x,\x)}\,\,,\quad
r_2(x,-\x)=-{r_2(x,\x)}\,,
\end{equation}
and, using \eqref{achille55} and \eqref{achille552}, we can note
\begin{equation}\label{cantolibero5}
|r_i |_{\mathcal{N}_{p}^{1}}\lesssim C\|u\|_{H^{p+s_0+3}}\,,\quad p+s_0+3\leq s\,,\;\;\;
i=1,2\,,
\end{equation}
for some $C>0$ depending on  $\|u\|_{H^{p+s_0+3}}$,
bounded
as $u$ goes to zero.
By the discussion above we deduce that
\begin{equation}\label{achille40}
\eqref{achille20}+\eqref{achille20B}=\ii E \opbw\left(
\begin{matrix}
c_2(x,\x) & d_2(x,\x)\vspace{0.2em}\\
\ov{d_2(x,-\x)} & c_2(x,\x)
\end{matrix}
\right)+\ii E\opbw\left(
\begin{matrix}
a_1^{(1)}(x,\x) & b_1^{(1)}(x,\x)\\
\ov{b_1^{(1)}(x,-\x)} & a_1^{(1)}(x,-\x)
\end{matrix}
\right)
\end{equation}
where 
\begin{equation}\label{cantolibero3bis}
a_1^{(1)}(x,\x):=a_1(x,\x)+c_1(x,\x)+r_1(x,\x)\,,
\quad b_1^{(1)}(x,\x):=r_2(x,\x)\,.
\end{equation}
up to 
a remainder satisfying \eqref{restofinale}
(see \eqref{achille19}, \eqref{timoria}, \eqref{timoria2}, \eqref{timoria3}, \eqref{achille31}, \eqref{achille31bis}).
The symbols $a_1^{(1)}(s,\x)$ and  $b_1^{(1)}(s,\x)$
satisfy the parity conditions \eqref{A1finale}
by \eqref{achille34}, \eqref{achille3444},
and the estimates
\eqref{A1finale3}
by Lemma \ref{realtaAAA}, and estimates   \eqref{achille35} and \eqref{cantolibero5}.
By \eqref{achille32}, \eqref{achille32bis}
we observe that
\begin{equation}\label{achille50}
\left(
\begin{matrix}
c_2(x,\x) & d_2(x,\x)\vspace{0.2em}\\
-\ov{d_2(x,-\x)} & -c_2(x,\x)
\end{matrix}
\right)=S^{-1}(x,\x)E(\uno+\widetilde{A}_{2,\mathtt{R}}(x,\x))S(x,\x)|\x|^{2}
\stackrel{\eqref{diagodiago}}{=}
\sm{\lambda_2(x,\x)}{0}{0}{-\lambda_2(x,\x)}|\x|^{2}\,.
\end{equation}
Therefore the \eqref{achille40}, \eqref{achille50}
implies the \eqref{systZZZ}.
This concludes the proof.
\end{proof}

\subsubsection{Block-diagonalization at order 1}
In this section we eliminate the off-diagonal symbol $b_1^{(1)}(x,\x)$
appearing in \eqref{systZZZ}, \eqref{A1finale}. 
In order to do this we consider the symbol
\begin{equation}\label{simboCCC}
c(x,\x):=\frac{\mathcal{X}_{\mathtt{R}}(\x)b_1^{(1)}(x,\x)}{2(|\x|^{2}+a_2^{(1)}(x,\x))}\,,
\qquad B(x,\x):=\left(
\begin{matrix}
0 & c(x,\x) \\ \ov{c(x,-\x)} & 0
\end{matrix}
\right)\,,
\end{equation}
where $a_2^{(1)}(x,\x), b_1^{(1)}(x,\x)$ 
are given by Proposition \ref{prop:block}
and $\mathcal{X}_{\mathtt{R}}(\x)$ is in \eqref{cutofftheta}.
We set
\begin{equation}\label{Mappa2}
\Phi_2(u)[\cdot]:=\uno+\opbw(B(x,\x))\,,\qquad \Psi_2(u)[\cdot]:=\uno-\opbw(B(x,\x)
-B^{2}(x,\x))\,,
\end{equation}
where $\uno:=\sm{1}{0}{0}{1}$ is the identity matrix. We have the following.

\begin{lemma}\label{propMAPPA2}
Assume the \eqref{regularity}.
For any $s\in \mathbb{R}$ 
the following holds:
\begin{enumerate}
\item[$(i)$] there exists a constant $C$ depending 
  on $\|u\|_{H^{2s_0+3}}$, bounded as $u$ goes to zero,  such that
\begin{equation}\label{stime-descentTOTA2}
\begin{aligned}
\|\Phi_2(u)V\|_{H^{s}}+\|\Psi_2(u)V\|_{H^{s}}&\leq 
\|{V}\|_{H^{s}}\big(1+C\|u\|_{H^{2s_0+3}}\big)\,,
\qquad \forall\, V\in H^{s}\,;
\end{aligned}
\end{equation} 

\item[$(ii)$] one has $\Psi_2(u)[\Phi_2(u)[\cdot]]=\uno+R_2(u)[\cdot]$
where $R_2$ is a real-to-real remainder 
satisfying 
\begin{equation}\label{achille102}
\|R_2(u)V\|_{H^{s+3}}\lesssim C \|V\|_{H^{s}}\|u\|_{H^{2s_0+7}}\,,
\end{equation}
\begin{equation}\label{achille102tris}
\|R_2(u)V\|_{H^{s+2}}\lesssim C \mathtt{R}^{-1}\|V\|_{H^{s}}\|u\|_{H^{2s_0+7}}\,,
\end{equation}
for some $C>0$ depending on $\|u\|_{H^{2s_0+7}}$,
bounded as $u$ goes to zero;
\item[$(iii)$] for $\mathtt{R}>0$ large enough with respect to 
$\mathtt{r}>0$ in \eqref{regularity}
the map $\uno+R_2(u)$ is invertible
and $(\uno+R_2(u))^{-1}=\uno+\widetilde{R}_2(u)$ with 
\begin{equation}\label{carmen2bis}
\|\widetilde{R}_2(u)V\|_{H^{s+2}}\lesssim C \mathtt{R}^{-1}\|V\|_{H^{s}}\|u\|_{H^{2s_0+7}}\,,
\end{equation}
for some $C>0$ depending on $\|u\|_{H^{2s_0+7}}$,
bounded as $u$ goes to zero. Moreover
$\Phi_2^{-1}(u):=(\uno+\widetilde{R}_2(u))\Psi_2(u)$ satisfies
\begin{equation}\label{carmenbis}
\|\Phi_2^{-1}(u)V\|_{H^{s}}\leq 
\|{V}\|_{H^{s}}\big(1+C\|u\|_{H^{2s_0+7}}\big)\,,
\qquad \forall\, V\in H^{s}(\mathbb{T}^{d};\mathbb{C})\,,
\end{equation}
for some $C>0$ depending on $\|u\|_{H^{2s_0+7}}$, bounded as $u$ goes to zero;

\smallskip
\item[$(iv)$] for almost any $t\in [0,T)$,
one has 
$\pa_{t}\Phi_2(u)[\cdot]=\opbw(\pa_{t}B(x,\x))$
and
\begin{equation}\label{achille112}
|\pa_{t}B(x,\x)|_{\mathcal{N}_{s_0}^{-1}}\lesssim_{\mathtt{r}} C \|u\|_{H^{2s_0+5}}\,,
\qquad
\|\pa_{t}\Phi_2(u)V\|_{H^{s+1}}\lesssim_{\mathtt{r}} C\|V\|_{H^{s}}\|u\|_{H^{2s_0+5}}\,,
\end{equation}
for some $C>0$ depending on $\|u\|_{H^{2s_0+5}}$, bounded as $u$ goes to zero.
\end{enumerate}
\end{lemma}

\begin{proof}
$(i)$ By \eqref{A1finale2}, \eqref{A1finale3} and \eqref{simboCCC} we deduce that
\begin{equation}\label{affari}
|c|_{\mathcal{N}^{-1}_{p}}\lesssim_{p} C\|u\|_{H^{p+s_0+3}}\,,\qquad p+s_0+3\leq s\,,
\end{equation}
for some $C>0$ depending on $\|u\|_{H^{p+s_0+3}}$,
bounded as $u$ goes to zero.
The bound \eqref{stime-descentTOTA2} follows by 
\eqref{actionSob} and \eqref{affari}.

\noindent
$(ii)$ By applying Lemma \ref{azioneSimboo}, 
Proposition \ref{prop:compo}, using \eqref{Mappa2}  and \eqref{affari}, 
we obtain the \eqref{achille102}. The \eqref{achille102tris} follows by item $(ii)$
of Proposition \ref{prop:compo}.

\noindent
$(iii)$
This item follows by using Neumann series, the \eqref{regularity}, bound
\eqref{achille102tris} and taking $\mathtt{R}$ large enough to obtain the smallness condition
$\|R_2(u)h\|_{H^{s+2}}\lesssim 1/2 \|V\|_{H^{s}}$.

\noindent
$(iv)$
This item follows by \eqref{simboCCC}, using the explicit formul\ae\,
\eqref{cantolibero3bis}, \eqref{cantolibero3}
and reasoning as in the proof of item $(iv)$
of Lemma \ref{propMAPPA}.
\end{proof}

We are ready to prove the following conjugation result.
\begin{proposition}{\bf (Block-diagonalization at order 1).}\label{prop:blockord1}
Assume \eqref{regularity}, consider the system \eqref{systZZZ}
and set (see \eqref{nuovavariabile})
\begin{equation}\label{nuovavariabile2}
W:=\Phi_2(u)[Z]\,.
\end{equation}
Then 
we have
\begin{equation}\label{systWWW}
\dot{W}=\ii E \opbw
\left(\begin{matrix}
|\x|^{2}+a_2^{(1)}(x,\x)+a_1^{(1)}(x,\x) &0 \vspace{0.3em}\\
0 & |\x|^{2}+a^{(1)}_2(x,\x)+{a_1^{(1)}(x,-\x) }
\end{matrix}
\right)W+\mathcal{R}_2(U)V
\end{equation}
where $a_2^{(1)}(x,\x), a_1^{(1)}(x,\x)$ are given in Proposition \ref{prop:block}
and
the remainder $\mathcal{R}_2$ is real to real and satisfies,
for any $s\geq 2s_0+7$,
 the estimate
\begin{equation}\label{restofinalebis}
\|\mathcal{R}_2(U)V\|_{H^{s}}\lesssim_{\mathtt{r}}  C\|V\|_{H^{s}}\|u\|_{H^{s}}\,,
\end{equation}
for some $C>0$ depending on $\|u\|_{H^{s}}$,
bounded as $u$ goes to zero. 
\end{proposition}

\begin{proof}
By \eqref{systZZZ} and \eqref{nuovavariabile2} we have
\[
\dot{W}=\Phi_2(u)\ii E\opbw\big(|\x|^{2}\uno+A_{2}^{(1)}(x,\x)+A_{1}^{(1)}(x,\x)\big)Z
+\Phi_2(u)\mathcal{R}(U)V+(\pa_{t}\Phi_2(u))Z\,.
\]
By item $(iii)$ of Lemma \ref{propMAPPA2}
we can write
$Z=\Phi_2^{-1}(u)W=(\uno+\widetilde{R}_2(u))\Psi_{2}(u)W$ 
 with $\widetilde{R}_2(u)$ satisfying \eqref{carmen2bis}.
Then
we have
\begin{equation}\label{achille15bis}
\dot{W}=\Phi_2(u)\ii E\opbw\Big(
|\x|^{2}\uno +A^{(1)}_{2}(x,\x)
+\mathcal{X}_{\mathtt{R}}(\x)A^{(1)}_{1}(x,\x)\Big)\Psi_2(u)W
+G(u)V\,,
\end{equation}
where (recall \eqref{nuovavariabile}, \eqref{nuovavariabile2})
\[
\begin{aligned}
G(u)&:=\Phi_2(u)
\ii E\opbw\Big(
|\x|^{2}\uno +A^{(1)}_{2}(x,\x)
+A^{(1)}_{1}(x,\x)\Big)
\widetilde{R}_2(u)
\Psi_2(u)\Phi_2(u)\Phi(u)\\
&+
\Phi_2(u)
\ii E\opbw\Big((1-\mathcal{X}_{\mathtt{R}}(\x))A^{(1)}_{1}(x,\x)\Big)
\Psi_2(u)\Phi_2(u)\Phi(u)
\\
&+\Phi_2(u)\mathcal{R}(U)+\pa_t\Phi_2(u)\Phi(u)\,.
\end{aligned}
\]
Using  Lemmata \ref{azioneSimboo}, \ref{propMAPPA}, \ref{propMAPPA2}
one can check that $G(u)$ satisfies the bound \eqref{restofinalebis}.
Reasoning similarly, and recalling \eqref{Mappa2}, we have that
\[
\Phi_2(u)\Big(\ii E\opbw\big(\mathcal{X}_{\mathtt{R}}(\x)
A^{(1)}_{1}(x,\x)\big)\Big)\Psi_2(u)=\ii E\opbw
\big(
\mathcal{X}_{\mathtt{R}}(\x)A^{(1)}_{1}(x,\x)\big)+G_1(u)
\]
for some $G_1(u)$ satisfying \eqref{restofinalebis}.
By Proposition \ref{prop:compo}
we deduce that
\begin{equation}\label{beatoporco}
\begin{aligned}
\Phi_2(u)\ii E\opbw&\Big(|\x|^{2}\uno+
A^{(1)}_{2}(x,\x)\Big)\Psi_2(u)=\\
&=\ii E\opbw\left( |\x|^{2}\uno+
A^{(1)}_{2}(x,\x)+
\left(\begin{matrix}0 & d(x,\x) \\
\ov{d(x,-\x)} & 0\end{matrix}\right)
 \right)+G_2(u)
 \end{aligned}
\end{equation}
where
\begin{equation}\label{beatoporco2}
d(x,\x):=-2c(x,\x)(|\x|^{2}+a_2^{(1)}(x,\x))\,,
\end{equation}
and where  $G_2(u)$ is some bounded remainder satisfying \eqref{restofinalebis}.
Using \eqref{simboCCC} we deduce that
\[
d(x,\x)+\mathcal{X}_{\mathtt{R}}(\x)b_1^{(1)}(x,\x)=0\,.
\]
By the discussion above (recall \eqref{achille15bis})
we have obtained the \eqref{systWWW}.
\end{proof}

\subsection{Proof of Theorem \ref{StimeEnergia}}

In this section we prove the energy estimate \eqref{StimeEnergia2}.
We first need some preliminary results.

\begin{lemma}\label{equivalenza}
Assume \eqref{regularity} 
and consider the function $W$ in \eqref{nuovavariabile2}.
For any $s\in \mathbb{R}$ 
one  has that
\begin{equation}\label{equivalenza2}
\|W\|_{H^{s}}\sim_{\mathtt{r}}\|V\|_{H^s}\,,
\end{equation}
where $\mathtt{r}$ is in \eqref{regularity}.
\end{lemma}
\begin{proof}
Recalling \eqref{nuovavariabile}, \eqref{nuovavariabile2} we write
$W=\Phi_2(u)Z=\Phi_2(u)\Phi(u)V$. By Lemmata \ref{propMAPPA}, \ref{propMAPPA2}
we also have $V=\Phi^{-1}(u)\Phi_2^{-1}(u)W$.
By estimates \eqref{stime-descentTOTA}, \eqref{stime-descentTOTA2},
\eqref{carmen} and \eqref{carmenbis} we have
\begin{equation*}
\|W\|_{H^{s}}\leq \|{V}\|_{H^{s}}\big(1+C\|u\|_{H^{2s_0+7}}\big)
\end{equation*}
\begin{equation*}
\|V\|_{H^{s}}\leq \|{W}\|_{H^{s}}\big(1+C\|u\|_{H^{2s_0+7}}\big)
\end{equation*}
for some constant $C$ depending on $\|u\|_{H^{2s_0+7}}$, 
bounded as $u$ goes to zero.
\end{proof}

Our aim is to estimate the norm of $V$ by using that $W$ 
solves the problem
\eqref{systWWW}. 
Let us define the symbol
\begin{equation}\label{LLL1}
\mathcal{L}:=\mathcal{L}(x,\x):=|\x|^{2}
+a_2^{(1)}(x,\x)\,,
\end{equation}
where $a_2^{(1)}(x,\x)$ is given in \eqref{A1finale}.
Notice that, by \eqref{A2tilde}, \eqref{nuovadiag}, we deduce
\begin{equation}\label{LLL1prop}
a_2^{(1)}(x,\x)\neq0 \quad \Rightarrow\quad |\x|\gtrsim\mathtt{R}\,.
\end{equation}

We now study some properties of the operator $T_{\mathcal{L}}$  
defined in \eqref{LLL1}.
\begin{lemma}\label{propLLL}
Assume the \eqref{regularity} and let $\gamma\in \mathbb{R}$, $\gamma>0$. 
Then, for 
$\mathtt{R}>0$ large enough (with respect to $\mathtt{r}>0$ in \eqref{regularity}), the
following holds true.

\noindent
$(i)$ The symbols $\mathcal{L}, (1+\mathcal{L})^{\pm\gamma}$
satisfy
\begin{equation}\label{LLLmeno1}
|\mathcal{L}|_{\mathcal{N}_{s_0}^{2}}+
|(1+\mathcal{L})^{\gamma}|_{\mathcal{N}_{s_0}^{2\gamma}}+|(1+\mathcal{L})^{-\gamma}|_{\mathcal{N}_{s_0}^{-2\gamma}}\leq
1+C\|u\|_{H^{2s_0+1}}\,,
\end{equation}
for some $C>0$ depending on $\|u\|_{H^{2s_0+1}}$,
bounded as $u$ goes to zero.

\noindent
$(ii)$
For any $s\in \mathbb{R}$ and any $h\in H^{s}(\mathbb{T}^{d};\mathbb{C})$, 
one has
\begin{equation}\label{stimaTL}
\|T_{\mathcal{L}}h\|_{H^{s-2}}+\|T_{(1+\mathcal{L})^{\gamma}}h\|_{H^{s-2\gamma}}
+\|T_{(1+\mathcal{L})^{-\gamma}}h\|_{H^{s+2\gamma}}
\leq \|h\|_{H^{s}}(1+C\|u\|_{H^{2s_0+1}})\,,
\end{equation}
\begin{equation}\label{stimaTLtris}
\|[{T}_{(1+\mathcal{L})^{\gamma}}, T_{\mathcal{L}}]h\|_{H^{s-2\gamma}}
\lesssim C \|h\|_{H^{s}}\|u\|_{H^{2s_0+5}}\,,
\end{equation}
for some $C>0$ depending on $\|u\|_{H^{2s_0+5}}$,
bounded as $u$ goes to zero.

\noindent
$(iii)$ 
One has that 
$T_{(1+\mathcal{L})^{-\gamma}}T_{(1+\mathcal{L})^{\gamma}}
=\uno+R(u)[\cdot]$
and, for any $s\in\mathbb{R}$ and any $h\in H^{s}(\mathbb{T}^{d};\mathbb{C})$,
\begin{align}
\|R(u) h\|_{H^{s+2}}&\lesssim C\|h\|_{H^{s}}\|u\|_{H^{2s_0+5}}\,,\label{LLLmeno2}\\
\|R(u) h\|_{H^{s+1}}&\lesssim C\mathtt{R}^{-1}\|h\|_{H^{s}}\|u\|_{H^{2s_0+5}}\,,\label{LLLmeno2bis}
\end{align}
for some $C>0$ depending on $\|u\|_{H^{2s_0+5}}$,
bounded as $u$ goes to zero.

\noindent
$(iv)$ For $\mathtt{R}\gg \mathtt{r}$ sufficiently large, 
the operator $T_{(1+\mathcal{L})^{\gamma}}$ 
has a left-inverse $T_{(1+\mathcal{L})^{\gamma}}^{-1}$.
For any $s\in\mathbb{R}$ one has
\begin{equation}\label{stimaLLLinv}
\|T_{(1+\mathcal{L})^{\gamma}}^{-1}h\|_{H^{s+2\gamma}}\leq
\|h\|_{H^{s}}(1+C\|u\|_{H^{2s_0+5}})\,,
\qquad \forall\, h\in H^{s}(\mathbb{T}^{d};\mathbb{C})\,,
\end{equation}
for some $C>0$ depending on $\|u\|_{H^{2s_0+5}}$,
bounded as $u$ goes to zero.

\noindent
$(v)$ 
For almost any $t\in [0,T)$ one have
\begin{equation}\label{LLLtempo}
|\pa_{t}a_2^{(1)}|_{\mathcal{N}^{2}_{s_0}}\lesssim_{\mathtt{r}} C\|u\|_{H^{2s_0+3}}\,.
\end{equation}
Moreover
\begin{equation}\label{LLLtempo2}
\|(T_{\pa_t(1+\mathcal{L})^{\gamma}})h\|_{H^{s-2\gamma}}\lesssim_{\mathtt{r}} C
\|h\|_{H^{s}}\|u\|_{H^{2s_0+3}}\,,
\qquad \forall\, h\in H^{s}(\mathbb{T}^{d};\mathbb{C})\,,
\end{equation}
for some $C>0$ depending on $\|u\|_{H^{2s_0+3}}$,
bounded as $u$ goes to zero.

\noindent
$(vi)$ The operators $T_{\mathcal{L}}$, 
$T_{\mathcal{L}^{-1}}$ are self-adjoint 
with respect to the $L^2$-scalar product \eqref{scalarL}.
\end{lemma}

\begin{proof}
$(i)$ It follows by \eqref{A1finale2} and \eqref{LLL1}.

\noindent
$(ii)$ The bound \eqref{stimaTL} follows by Lemma \ref{azioneSimboo} 
and \eqref{LLLmeno1}. Let us check the \eqref{stimaTLtris}.
By Proposition \ref{prop:compo} we deduce that 
(recall formul\ae\, \eqref{PoissonBra}, \eqref{PoissonBra22})
\[
[{T}_{(1+\mathcal{L})^{\gamma}}, T_{\mathcal{L}}]=
\opbw\big(\frac{1}{\ii }\big\{(1+\mathcal{L})^{\gamma}, \mathcal{L}\big\}\big)+R(u)
\]
where the remainder $R(u)$ satisfies (see \eqref{composit2} and \eqref{LLLmeno1})
\[
\|R(u)h\|_{H^{s-2\gamma+1}}\lesssim C\|h\|_{H^{s}}\|u\|_{2s_0+5}\,,
\]
for some $C>0$ depending on $\|u\|_{2s_0+5}$, bounded as $u$ goes to zero.
By an explicit computation, recalling \eqref{LLL1} 
one gets $\{(1+\mathcal{L})^{\gamma},\mathcal{L}\}=0$.
This implies the \eqref{stimaTLtris}.

\noindent
$(iii)$ This item follows by applying Proposition \ref{prop:compo} and using that
  $\{(1+\mathcal{L})^{-\gamma}, (1+\mathcal{L})^{\gamma}\}=0$.
 The bound \eqref{LLLmeno2} follows by \eqref{prodSimboli} 
 (with $a\rightsquigarrow(1+\mathcal{L})^{-\gamma}$, 
 $b\rightsquigarrow(1+\mathcal{L})^{\gamma}$), \eqref{actionSob}, 
 \eqref{LLLmeno1} and \eqref{composit2}.
 The \eqref{LLLmeno2bis}  follows by \eqref{composit2Theta}
 using the \eqref{LLL1prop}.
 
 \noindent
 $(iv)$ To prove this item
we use Neumann series. We define (using item $(iii)$)
\[
T_{(1+\mathcal{L})^{\gamma}}^{-1}
:=(\uno+R(u))^{-1}T_{(1+\mathcal{L})^{-\gamma}}\,,
\qquad (\uno+R(u))^{-1}:=\sum_{k=0}^{\infty}(-1)^{k}(R(u))^{k}\,.
\]
Using \eqref{LLLmeno2bis} and taking $\mathtt{R}$ sufficiently large
 one can check that 
\[
\|(\uno+R(u))^{-1}h\|_{H^{s}}\lesssim \|h\|_{H^{s}}(1+C\|u\|_{H^{2s_0+5}})\,,
\]
for some $C>0$ depending on $\|u\|_{H^{2s_0+5}}$, bounded as $u$ goes to zero.
The bound above together with \eqref{stimaTL}
implies the \eqref{stimaLLLinv}.

\noindent
$(v)$ By \eqref{nuovadiag} we have
\[
\pa_{t}\tilde{a}_{2,\mathtt{R}}^{+}=\frac{1}{2\lambda_2}\big(
2(1+\widetilde{a}_{2,\mathtt{R}})\pa_t\widetilde{a}_{2,\mathtt{R}}
- \pa_{t}\widetilde{b}_{2,\mathtt{R}}
\ov{\widetilde{b}_{2,\mathtt{R}}}-\widetilde{b}_{2,\mathtt{R}}
\pa_{t}\ov{\widetilde{b}_{2,\mathtt{R}}}
\big)\,.
\]
Moreover, recalling \eqref{A2tilde}, \eqref{simboa2}, the hypotheses of Lemma
\ref{lem:nonomosimbo} are satisfied. Therefore, using \eqref{simboStima2} 
and \eqref{regularity}, we deduce
\[
|\pa_{t}\widetilde{a}_{2,\mathtt{R}}|_{\mathcal{N}_{s_0}^{0}}\lesssim_{\mathtt{r}} C\|u\|_{H^{2s_0+3}}\,.
\]
Similarly one can prove the same estimate for $\widetilde{b}_2$.
Hence the \eqref{LLLtempo} follows.
The \eqref{LLLtempo} and \eqref{actionSob} imply the \eqref{LLLtempo2}.

\noindent
$(vi)$
Item $(vi)$ follows by \eqref{simboAggiunto}
since $\mathcal{L}, \mathcal{L}^{-1}$ are real valued.
\end{proof}

In the following we shall construct the \emph{energy norm}. By using this norm we are able to achieve the \emph{energy estimates} on the previously diagonalized system. This \emph{energy norm} is equivalent to the Sobolev one.
For $s\in \mathbb{R}$, $s\geq2s_0+7 $ we define 
\begin{equation}\label{def:VVTL}
w_{\gamma}:=T_{(1+\mathcal{L})^{\gamma}}w\,,\qquad 
W_{\gamma}:=\vect{w_\gamma}{\ov{w_{\gamma}}}
=T_{(1+\mathcal{L})^{\gamma}}\uno W\,,\quad W=\vect{w}{\bar{w}}\,,\quad \gamma:=\frac{s}{2}\,.
\end{equation}

\begin{lemma}[Equivalence of the energy norm]\label{bowie:lemm}
Assume \eqref{regularity}. Then,  for  $\mathtt{R}>0$ large enough, one has
\begin{equation}\label{equivNorme}
\|w_{\gamma}\|_{L^{2}}\sim_{\mathtt{r}}\|w\|_{H^{s}}\,,
\end{equation}
where $\mathtt{r}$ is given in \eqref{regularity}.
\end{lemma}
\begin{proof}
It follows by using estimates \eqref{stimaTL}, \eqref{stimaLLLinv}
and reasoning as in the proof of Lemma \ref{equivalenza}.
\end{proof}

Notice that, by using Lemma \ref{azioneSimboo} (see \eqref{azionecufoffoff}) 
and by \eqref{restofinalebis}, 
the \eqref{systWWW} is equivalent to (recall \eqref{LLL1})
\begin{equation}\label{systZZZBis}
\pa_{t}w=\ii T_{\mathcal{L}}w+\ii T_{a_{1,\mathtt{R}}^{(1)}}w+\mathcal{Q}_1(u)v+
\mathcal{Q}_2(u)\bar{v}\,,\quad W:=\vect{w}{\bar{w}}\,,
\end{equation}
where 
\begin{equation}\label{QQQ1}
\begin{aligned}
&a_{1,\mathtt{R}}^{(1)}(x,\x):=a_{1}^{(1)}(x,\x)\mathcal{X}_{\mathtt{R}}(\x)\,,\\
&\|\mathcal{Q}_i(u)h\|_{H^{s}}\lesssim_{\mathtt{r}} C\|h\|_{H^{s}}\|u\|_{H^{s}}\,,\;\;\; \forall \, h\in H^{s}(\mathbb{T}^{d};\mathbb{C})\,,\; i=1,2\,,\;\; s\geq2s_0+7\,,
\end{aligned}
\end{equation}
for some constant $C>0$ depending on $\|u\|_{H^{s}}$,
bounded as $u$ goes to zero. 

\begin{lemma}\label{equaVVnn}
Recall \eqref{systZZZBis}.
One has that the function $w_{\gamma}$ defined in \eqref{def:VVTL}
solves the problem
\begin{equation}\label{equa:VVTL}
\pa_{t}w_{\gamma}=\ii T_{\mathcal{L}}w_{\gamma}
+\ii \mathcal{A}_{\gamma}(u)w_{\gamma}+\mathcal{B}_{\gamma}(u)w_{\gamma}+\mathcal{R}_{\gamma}(u)[V]\,,
\end{equation}
where 
\begin{align}
&\mathcal{A}_{\gamma}(u):=T_{(1+\mathcal{L})^{\gamma}}
T_{a_{1,\mathtt{R}}^{(1)}}(T_{(1+\mathcal{L})^{\gamma}})^{-1}\,,\qquad 
\mathcal{B}_{\gamma}(u):=
T_{\pa_{t}(1+\mathcal{L})^{\gamma}}
(T_{(1+\mathcal{L})^{\gamma}})^{-1}\,,\label{AAAn}
\end{align}
and where $\mathcal{R}_{\gamma}$ satisfies
\begin{equation}\label{stimeRRRN}
\|\mathcal{R}_{\gamma}(u)V\|_{L^{2}}\lesssim_{\mathtt{r}} C\|V\|_{H^{s}} \|u\|_{H^{s}}\,,
\end{equation}
for some $C>0$ depending on $\|u\|_{H^{s}}$, 
bounded as $u$ goes to zero.
\end{lemma}

\begin{proof}
By differentiating \eqref{def:VVTL} and using \eqref{systZZZBis}
we get the \eqref{equa:VVTL} with $\mathcal{A}_{\gamma}(u)$, $\mathcal{B}_{\gamma}(u)$
as in \eqref{AAAn} up to a remainder $\mathcal{R}_{\gamma}$.
%
The estimate \eqref{stimeRRRN} follows by 
\eqref{stimaTLtris}, \eqref{stime-descentTOTA2},
\eqref{stime-descentTOTA}, \eqref{stimaTL}
and \eqref{QQQ1}.
\end{proof}

\begin{proof}[{\bf Proof of Theorem \ref{StimeEnergia}}]
We estimate the $L^{2}$-norm
of $w_{\gamma}$ satisfying \eqref{equa:VVTL}.
Recalling  \eqref{scalarL}, we have
\begin{equation}\label{napoli3}
\begin{aligned}
\pa_{t}\|w_{\gamma}\|_{L^{2}}^{2}&\sim
{\rm Re}(\ii \mathcal{A}_{\gamma}(u)w_{\gamma},w_{\gamma})_{L^{2}}+
{\rm Re}( \mathcal{B}_{\gamma}(u)w_{\gamma},w_{\gamma})_{L^{2}}+
{\rm Re}( \mathcal{R}_{\gamma}(u)V,w_{\gamma})_{L^{2}}\,,
\end{aligned}
\end{equation}
where we have used that ${\rm Re}(\ii T_{\mathcal{L}}w_{\gamma},w_{\gamma})_{L^{2}}=0$ thanks to item $(vi)$ of Lemma  \ref{propLLL}.
We analyze each summand separately.
First of all we note that
\[
\|\mathcal{B}_{\gamma}(u)w_{\gamma}\|_{L^{2}}
\stackrel{\eqref{LLLtempo2}, \eqref{stimaLLLinv}}{\lesssim_{\mathtt{r}}}
C
\|u\|_{H^{2s_0+5}}\|w_{\gamma}\|_{L^{2}}\,,
\]
for some constant $C>0$ depending on $\|u\|_{H^{2s_0+5}}$,
bounded as $u$ goes to zero.
Hence, by Cauchy-Swartz inequality, we obtain
\begin{equation}\label{napoli4}
{\rm Re}( \mathcal{B}_{\gamma}(u)w_{\gamma},w_{\gamma})_{L^{2}}\lesssim_{\mathtt{r}} C
\|u\|_{H^{2s_0+5}}\|w_{\gamma}\|_{L^{2}}^{2}\,.
\end{equation}
Using \eqref{stimeRRRN} we obtain
\begin{equation}\label{napoli5}
{\rm Re}( \mathcal{R}_{\gamma}(u)V,w_{\gamma})_{L^{2}}\lesssim_{\mathtt{r}} C
\|u\|_{H^{s}}\|w_{\gamma}\|_{L^{2}}\|V\|_{H^{s}}\,,
\end{equation}
for $s\geq 2s_0+7$ and for some $C>0$ depending on  
$\|u\|_{H^{s}}$, bounded as $u$ goes to zero.
We now study the most difficult term, i.e. the one depending on $\mathcal{A}_{\gamma}$.
We write
\begin{equation}\label{napoli7}
\mathcal{A}_{\gamma}(u)=T_{a_{1,\mathtt{R}}^{(1)}}
+\mathcal{C}_{\gamma}(u)\,,
\qquad \mathcal{C}_{\gamma}(u):=
[ T_{(1+\mathcal{L})^{\gamma}},T_{a_{1,\mathtt{R}}^{(1)}}]
(T_{(1+\mathcal{L})^{\gamma}})^{-1}\,.
\end{equation}
By applying Proposition \ref{prop:compo} and using estimates
\eqref{A1finale3}, \eqref{LLLmeno1}, \eqref{composit2} and \eqref{stimaLLLinv}
we obtain
\begin{equation}\label{napoli11}
\|\mathcal{C}_{\gamma}(u)w_{\gamma}\|_{L^{2}}
\lesssim
C
\|u\|_{H^{2s_0+7}}\|w_{\gamma}\|_{L^{2}}\,,
\end{equation}
for some constant $C>0$ depending on $\|u\|_{H^{2s_0+7}}$,
bounded as $u$ goes to zero.
Recall that the symbol $a_{1,\mathtt{R}}^{(1)}$ is 
real valued (see \eqref{A1finale}), then
the operator $T_{a_{1,\mathtt{R}}^{(1)}}$ is self-adjoint 
w.r.t. the scalar product \eqref{scalarL}.
As a consequence we have
\begin{equation}\label{napoli12}
{\rm Re}(\ii \mathcal{A}_{\gamma}(u)w_{\gamma},w_{\gamma})_{L^{2}}
\stackrel{\eqref{napoli7}}{=}
{\rm Re}(\ii \mathcal{C}_{\gamma}(u)w_{\gamma},w_{\gamma})_{L^{2}}
\stackrel{\eqref{napoli11}}{\lesssim}C
\|w_{\gamma}\|_{L^{2}}^{2}
\|u\|_{H^{2s_0+7}}\,.
\end{equation}
By \eqref{napoli3}, \eqref{napoli4}, \eqref{napoli5} and \eqref{napoli12}
we get 
\begin{equation}\label{napoli32}
\pa_{t}\|w_{\gamma}\|_{L^{2}}^{2}\lesssim_{\mathtt{r}} C
\|w_{\gamma}\|_{L^{2}}^{2}
\|u\|_{H^{2s_0+7}}+C\|u\|_{H^{s}}\|w_{\gamma}\|_{L^{2}}\|V\|_{H^{s}}\,.
\end{equation}
By 
\eqref{equivNorme}, \eqref{equivalenza2} the 
\eqref{napoli32} becomes
\[
\pa_{t}\|w_{\gamma}\|_{L^{2}}^{2}\lesssim_{\mathtt{r}} C
\|u\|_{H^{s}}\|v\|^{2}_{H^{s}} \quad \Rightarrow
\quad
\|w_{\gamma}\|_{L^{2}}^{2}\leq \|w_{\gamma}(0)\|_{L^{2}}^{2}+\int_{0}^{t}C
\|u(\tau)\|_{H^{s}}\|v(\tau)\|^{2}_{H^{s}} d\tau\,,
\]
where $C>0$ depends on $\|u\|_{H^{s}},\mathtt{r}$.
By using again the equivalences
\eqref{equivNorme}, \eqref{equivalenza2}
we get the \eqref{StimeEnergia2}.
\end{proof}

In the following we prove the existence of the solution of a linear problem of the form
\begin{equation}\label{QLNLSV2020}
\left\{\begin{aligned}
&\dot{V}=\ii E\opbw\big(|\x|^{2}\uno +A_{2}(x,\x)+A_{1}(x,\x)\big)V+R_1(U)V+R_2(U)U\,,\\
&V(0)=V_0:=U(0)\,,
\end{aligned}\right.
\end{equation}
where the para-differential part is assumed to be like in system 
\eqref{QLNLSV} and $R_2(U)U$ has to be considered  as a forcing term, 
the function $U$ satisfies \eqref{regularity} and 
the operators $R_1$ and $R_2$ are bounded.

\begin{lemma}\label{esistenza-gen} Let $s\geq 2(d+1)+7$.
Consider the problem \eqref{QLNLSV2020}, 
assume \eqref{regularity} and 
that the matrices $A_2$, $A_1$ are like the 
ones in system \eqref{QLNLSV}. 
Assume moreover that $R_i$ are real-to-real and
satisfy \eqref{stimaRRR} for $i=1,2$.
  Then there exists a unique solution $V$ of \eqref{QLNLSV} in $ L^{\infty}([0,T); H^{s}(\T^d;\C^2))\cap Lip([0,T); H^{s-2}(\T^d;\C^2))\cap\mathcal{U}$, moreover it satsfies the following estimate
  \begin{equation}\label{stimalineareforza}
  \|V(t)\|_{H^s}\leq \mathtt{C}_1 e^{\mathtt{C}T}\big((1+\mathtt{C}T)
  \|V_0\|_{H^s}+\mathtt{C}T\|U\|_{L^{\infty}([0,T);H^s)}\big)\,,
    \end{equation}
for positive constants $\mathtt{C}_1>0$ depending on $\mathtt{r}$
 in \eqref{regularity} and $\mathtt{C}>0$ depending on
$\|u\|_{H^{s}}, \mathtt{r}$ 
and bounded from above as
$\|u\|_{H^{s}}$ goes to zero.
\end{lemma}
\begin{proof}
Let us consider first the case of the free equation, i.e. we assume for the moment $R_1(U)V=R_2(U)U=0$.
For any $\lambda\geq1$ we consider the following localized matrix
\begin{equation}\label{uomomascherato}
A^{\lambda}(x,\xi):=\Big(
A_2(x,\xi)+A_1(x,\xi)\Big)\chi\Big(\frac{\xi}{\lambda}\Big)\,,
\end{equation}
where $\chi$ is a cut-off function whose support is contained in the ball of center 0 and radius 1. 
Note that the quantities $|A^{\lambda}|_{\mathcal{N}_{s}^{2}}$ 
and $ \lambda^{-2}|A^{\lambda}|_{\mathcal{N}_{s}^{0}}$ 
(where we meant the semi-norm \eqref{normaSimbo} of each entry)
are uniformly bounded in $\lambda\geq 1$.

Let $V_{\lambda}$ be the solution of the following linear Schr\"odinger equation
\begin{equation*}
\begin{cases}
\dot{V_{\lambda}}=\ii E\opbw(|\x|^{2}\uno+A^{\lambda}(x,\xi))V_{\lambda}\\
V_{\lambda}(0)=V_0\,.
\end{cases}
\end{equation*}
For initial data in $H^{s}$ the existence follows by the Duhamel formulation of the problem, 
a classical fixed point argument and using that $ \lambda^{-2}|A^{\lambda}|_{\mathcal{N}_{s}^{0}}$
is uniformly bounded in $\lambda\geq1$ (in other words $\opbw(A^{\lambda})$ is a bounded perturbation with estimates depending on $\lambda$).
The function $V_{\lambda}$ is  continuous with values in $H^{s}$ with estimates which, \emph{a priori}, 
depend on $\lambda$.
We claim that the following estimate, with constants  independent of $\lambda$, holds true:
\begin{equation}\label{bowie}
\|V_{\lambda}(t)\|_{H^s}^2\lesssim_{\mathtt{r}} 
\|V_0\|_{H^s}^2+\int_0^t C\|U(\sigma)\|_{H^s}\|V_{\lambda}(\sigma)\|_{H^s}^2d\sigma\,,
\end{equation}
for some $C>0$ depending $\|U\|_{H^{s}}, \mathtt{r}$ 
and bounded as $U$ goes to $0$.
The proof of the claim follows the lines of Theorem \ref{StimeEnergia}
with $A(x,\x)=A_2(x,\xi)+A_1(x,\xi)$ in \eqref{QLNLSV} replaced by  $A^{\lambda}$
in \eqref{uomomascherato}, we recall below the fundamental steps.
One has to diagonalize the matrix $E(|\x|^{2}\uno+A^{\lambda})$ as done in 
Propositions \ref{prop:block}, \ref{prop:blockord1}
obtaining a new diagonal matrix of order two of the form
\[
\left(\begin{matrix}
|\x|^{2}+a_2^{(1,\lambda)}(x,\x)+a_1^{(1,\lambda)}(x,\x) &0 \vspace{0.3em}\\
0 & |\x|^{2}+a^{(1,\lambda)}_2(x,\x)+{a_1^{(1,\lambda)}(x,-\x) }
\end{matrix}
\right)
\]
where $|a_j^{(1,\lambda)}(x,\x)|_{\mathcal{N}_{s}^{j}}$ are uniformly bounded in $\lambda$ for
$j=1,2$.
At this point, following Lemmata \ref{bowie:lemm}, \ref{equaVVnn},
one obtains the energy estimate \eqref{bowie}
by defining the new variable as in \eqref{def:VVTL}
with $\mathcal{L}$ replaced by $\mathcal{L}^{\lambda}:=|\x|^{2}+a_2^{(1,\lambda)}(x,\x)$.
The constant in \eqref{bowie} is independent of $\lambda$ 
because $|\mathcal{L}^{\lambda}|_{\mathcal{N}_s^{2}}$ is uniformly bounded in $\lambda$.

\noindent
As a consequence of \eqref{bowie} and  Gr\"onwall lemma we have
\begin{equation*}
\|V_{\lambda}(t)\|^2_{H^s}\lesssim_{\mathtt{r}} \| V_0\|_{H^s}^2 \exp\left(\int_{0}^tC\|U(\sigma)\|_{H^s}d\sigma\right),
\end{equation*}
which gives the uniform boundedness of the family in 
$L^{\infty}([0,T);H^{s}(\T^d;\mathbb{C}^{2}))$. This implies that the family is uniformly bounded in $C^0([0,T);H^{s}(\T^d;\C^2))$.
Similarly, by using also \eqref{QLNLSV2020}, one proves that the family is uniformly 
Lipschitz continuous in the space 
$C^0([0,T);H^{s-2}(\T^d;\mathbb{C}^{2}))$, 
therefore one gets, up to subsequences, by Ascoli-Arzel\`a theorem, 
a limit $\Phi(t)U(0)$ in the latter space 
which is a solution of equation \eqref{QLNLSV2020} 
with $R_1(U)V=R_2(U)U=0$. The limit $\Phi(t)U(0)$ is Lipschitz continuous with values in $H^{s-2}$.
Moreover by using \eqref{StimeEnergia2} and Gr\"onwall lemma one obtains
\begin{equation}\label{Gronwall}
\|\Phi(t)V_0\|^2_{H^s}\lesssim_{\mathtt{r}} \| V_0\|_{H^s}^2 \exp\left(\int_{0}^tC\|U(\sigma)\|_{H^s}d\sigma\right),
\end{equation}
i.e. $\|\Phi(t)V_0\|_{H^s}\lesssim_{\mathtt{r}} \| V_0\|_{H^s} e^{\mathtt{C}t}$ 
for some $\mathtt{C}$ 
depending on $\|U\|_{H^{s}}, \mathtt{r}$
 bounded as $U$ goes to $0$. 
The solution $\Phi(t)V_0$ is in $\mathcal{U}$ since $V_0\in \mathcal{U}$
and the matrix of symbols
$A^{\lambda}$ in \eqref{uomomascherato} is real-to-real.

To prove the existence of the solution in the case that $R_1(U)V$ and $R_{2}(U)U$ are non zero we use Picard iterates and we reason as follows.
We define the operator
\begin{equation*}
T(W):=\Phi(t)V_0+\Phi(t)\int_0^t[\Phi(\sigma)]^{-1}\big(R_1(U)W(\sigma)+R_2(U)U(\sigma)\big)d\sigma
\end{equation*}
and the sequence 
\begin{equation*}
\begin{cases}
W_0=\Phi(t) U(0)\\
W_n=T(W_{n-1})\,.
\end{cases}
\end{equation*}
In this way we obtain that 
$\|W_{n+1}-W_n\|_{H^{s}}\leq 
\frac{(\mathtt{C}T)^n}{n!}\|W_1-W_0\|_{H^{s}}$. 
We find a fixed point for the operator 
$T$ as $V=\sum_{n=1}^{\infty} W_{n+1}-W_n +W_0 $, 
the estimate \eqref{stimalineareforza} 
may be obtained by direct computation 
from the definition of  the solution $W$. 
\end{proof}

\begin{remark}\label{stima-migliore}
If $R_2(U)U=0$ in the previous lemma, one gets the better estimate
\begin{equation}\label{bitter}
\|V(t)\|_{H^{s}}\leq \mathtt{C}_1(1+\mathtt{C}T) e^{T\mathtt{C}}\|V_0\|_{H^s}.
\end{equation}
\end{remark}
\section{Proof of the main Theorem \ref{main}}\label{sec:5}
The proof of the Theorem \ref{main} relies on the 
iterative scheme which is described below. 
We recall that, by Proposition \ref{NLSparapara}, the equation \eqref{QLNLS}  
is equivalent to the para-differential system 
\eqref{QLNLS444}.  
We consider the following sequence of Cauchy problems 
\begin{equation*}
\mathcal{P}_1=\begin{cases}
 \partial_t U_1=\ii E\Delta U_1\\
 U_1(0,x)=\tilde{U}_0(x),
\end{cases}
\end{equation*}
where $\tilde{U}_0(x)=(\tilde{u}_0,\bar{\tilde{u}}_0)$ 
is the initial condition of \eqref{QLNLS},
and we define by induction
\begin{equation*}
\mathcal{P}_n=\begin{cases}
 \partial_t U_n=\ii E \opbw(|\xi|^2\uno +A_2(U_{n-1};x,\xi)+A_1(U_{n-1};x,\xi))U_n+R(U_{n-1})U_{n-1}\\
 U_n(0,x)=\tilde{U}_0(x)\,.
\end{cases}
\end{equation*}
In the following lemma we prove that the sequence is well defined, 
moreover the sequence of solutions 
$\{U_n\}_{n\in\N}$ is bounded in 
$H^{s}(\mathbb{T}^{d};\mathbb{C}^{2})$ and converging in 
$H^{s-2}(\mathbb{T}^{d};\mathbb{C}^{2})$.
\begin{lemma}\label{lemma-iterativo}
Fix $\tilde{U}_0\in H^{s}(\mathbb{T}^{d};\mathbb{C}^{2})\cap\mathcal{U}$ 
such that $\|\tilde{U}_0\|_{H^{s}}\leq r$ with $s\geq 2( d+1)+9$, 
then there exists a time $T>0$ 
small enough such that the following holds true.
For any $n\in\N$ the problem $\mathcal{P}_n$ 
admits a unique solution 
$U_n$ in 
$L^{\infty}([0,T);H^s(\T^d;\mathbb{C}^{2}))\cap Lip([0,T);H^{s-2}(\T^d;\mathbb{C}^2))$. 
Moreover it satisfies the following conditions:

\noindent
$(S1)_n$: There exists a constant $C$ depending on $s$ and $\|\tilde{U}_0\|_{H^{s-2}}$ 
such that for any 
$1\leq m\leq n$ one has 
\[
\| U_m\|_{L^{\infty}([0,T);H^s)}\leq \Theta\,, \qquad \Theta:=C r\,.
\]

\noindent
$(S2)_n$: For $1\leq m\leq n$ one has $\|U_m-U_{m-1}\|_{L^{\infty}([0,T),H^{s-2})}
\leq 2^{-m+1}\|\tilde{U}_0\|_{H^{s-2}}$, where we have defined $U_0=0$.
\end{lemma}
\begin{proof}
The proof of $(S1)_1$ and $(S2)_{1}$ is trivial, 
let us suppose that $(S1)_{n-1}$ and $(S2)_{n-1}$ hold true. 
We prove $(S1)_{n}$ and $(S2)_{n}$. 
We first note that $\|U_{n-1}\|_{L^{\infty}H^{s-2}}$ 
does not depend on $\Theta$. 
Indeed by using $(S2)_{n-1}$ one proves that  
$\|U_{n-1}\|_{L^{\infty}H^{s-2}}\leq 2\|\tilde{U}_0\|_{H^{s-2}}$. 
Therefore Lemma \ref{esistenza-gen} applies with $\mathtt{r}$ chosen large enough
with respect to $\|\tilde{U}_{0}\|_{H^{s-2}}$, and 
the constants $\mathtt{C}, \mathtt{C}_1$ therein  depend on $\Theta$ and 
$\|\tilde{U}_0\|_{H^{s-2}}$ respectively.
We have the estimate
\[  
\|U_{n}(t)\|_{H^{s}}\leq \mathtt{C}_1 e^{\mathtt{C}T}((1+\mathtt{C}T)\|\tilde{U}_0\|_{H^s}
+T\mathtt{C}\|U_{n-1}\|_{L^{\infty}H^s})\,.
\]
To prove the $(S1)_n$ we need to impose the bound 
\begin{equation}\label{sceltaT}
\mathtt{C}_1 e^{\mathtt{C}T}((1+\mathtt{C}T)\|\tilde{U}_0\|_{H^s}
+T\mathtt{C}\Theta)\leq\Theta\,,
\end{equation}
this is possible by choosing 
$\mathtt{C}_1 T\mathtt{C}\leq 1/8$ and 
$8 \mathtt{C}_1 \|\tilde{U}_0\|_{H^{s}}\leq \Theta/2$. 

\noindent Let us prove $(S2)_n$.
We use the notation 
$A(U\,;x,\xi):=|\xi|^2\uno+ A_2(U\,;x,\xi)+A_1(U\,;x,\xi)$ 
and $V_{n}:=U_n-U_{n-1}$. 
The function $V_n$ solves the equation
\begin{equation}\label{s2n}
\partial_t V_n=\ii E\opbw(A(U_{n-1};x,\xi))V_n+f_n\,,
\end{equation}
where 
\begin{equation*}
f_n=\ii E\opbw\big(A(U_{n-1};x,\xi)-A(U_{n-2};x,\xi)\big)U_{n-1}+R(U_{n-1})U_{n-1}-R(U_{n-2})U_{n-2}\,.
\end{equation*}
The equation \eqref{s2n} with $f_n=0$ admits a well-posed flow 
$\Phi(t)$ thanks to Lemma \ref{esistenza-gen}, 
moreover it satisfies the \eqref{bitter}. 
Therefore by Duhamel principle we have
\begin{equation*}
\|V_{n}\|_{H^{s-2}}\leq 
\Big\|\Phi(t)\int_0^t(\Phi(\sigma))^{-1}f_n(\sigma) d\sigma\Big\|_{H^{s-2}}
\leq \mathtt{C}_1^{2}(1+\mathtt{C}T)^2e^{2T\mathtt{C}}\int_0^t\|{f_n}(\sigma)\|_{H^{s-2}}d\sigma\,.
\end{equation*}
Using the Lipschitz estimates on the matrices $A$ (which may be deduced by \eqref{simboStima2})
and $R$ (see the bound \eqref{nave101}),
and the inductive hypothesis one proves that 
$\|f_n\|_{H^{s-2}}\leq \mathtt{C}_2\|V_{n-1}\|_{H^{s-2}} $ 
for a positive constant $\mathtt{C}_2$ depending 
on $\Theta$ and $s$. 
Therefore by induction, using the choice of $T$ done in \eqref{sceltaT}, one obtains $\|V_n(t)\|_{H^{s-2}}\leq \frac{(Kt)^n}{n!}\|\tilde{U}_0\|_{H^{s-2}}$ for some $K$ independent on $n$,  depending on $\mathtt{C}_2$ and hence on   $\Theta$ and $s$. By choosing $T$ in such a way that $KT<1$ one obtains the $(S2)_n$.
\end{proof}
We are now in position to prove the Theorem \ref{main}.
\begin{proof}[{\bf Proof of Theorem \ref{main}}] 
Fix $s> 2(d+1)+9$.
We first prove the existence of a weak solution of the Cauchy problem, then we prove that it is actually continuous and unique, finally we prove the continuity of the solution map.

{\bf Weak solutions.} 
From Proposition \ref{NLSparapara} we know that equation 
\eqref{QLNLS} is equivalent to \eqref{QLNLS444}. 
We consider the sequence of problems 
$\mathcal{P}_n$ previously defined. 
From Lemma \ref{lemma-iterativo} we obtain a sequence 
of solutions $U_n$ which is bounded in 
$L^{\infty}([0,T);H^s(\T^d;\mathbb{C}^{2}))$, 
by a direct computation one proves also that the sequence 
$\partial_tU_n$ is bounded in $L^{\infty}([0,T);H^{s-2}(\T^d;\mathbb{C}^{2}))$. 
Thus 
we get a weak-* limit 
$U\in L^{\infty}([0,T);H^s(\T^d;\mathbb{C}^{2}))\cap 
Lip([0,T);H^{s-2}(\T^d;\mathbb{C}^{2}))$,
satisfying 
\begin{equation}\label{hystoria}
\|U\|_{L^{\infty}H^{s}}\leq \Theta=C \|\tilde{U}_0\|_{H^{s}}\,,
\end{equation}
where $C>0$ is some constant depending on $\|\tilde{U}_0\|_{H^{s-2}}$.
In order to show that the limit $U$ solves the equation 
it is enough to prove that it solves it in the sense of distribution. 
One can check that 
\begin{equation*}
\|\opbw(A(U;x,\xi))U+R(U)U-\opbw(A(U_{n-1};x,\xi))U_n
-R(U_{n-1})U_{n-1}\|_{H^{s-4}}
\end{equation*}
goes to zero when $n$ goes to $\infty$, this is a 
consequence of triangular inequality, Lipschitz estimates on the matrix $A$ and $R$
and Lemma \ref{lemma-iterativo} 
(in particular the boundedness of $U_{n}$ in $H^s(\T^d;\C^2)$ 
and the strong convergence in $H^{s-2}(\T^d;\C^2)$). 
\begin{remark}
The fact that the time of existence $T>0$
 depends only on $\|\tilde{U}_0\|_{H^{s}}$
 is a consequence of Lemma \ref{lemma-iterativo}.
 One could show that, in the case of small initial conditions of size $0<\e<1$, the time $T$
 would be of order $O(\e^{-1})$.
 \end{remark}
   
{\bf Strong solutions.} 
In order to prove that $U$ is in the space 
$C^0([0,T);H^s(\T^d;\mathbb{C}^{2}))$ 
we show that it is the strong limit of function in 
$C^0([0,T);H^s(\mathbb{T}^{d};\mathbb{C}^{2}))$. 
We consider the following  smoothed  version of the initial condition 
\begin{equation}\label{tronco}
V^N_0(x):=S_{\leq N}V_0(x):=(1-S_{> N})V_0(x):=\sum_{|k|\leq N}(V_{0})_ke^{\ii k\cdot x}\,,
\end{equation}
and we define $U^N$ the solution of \eqref{QLNLS444} with initial condition $V^N_0$. The $U^N$, since 
$V^N_0$ is $C^{\infty}$ 
(in particular $H^{s+2}$), 
are in $C^{0}([0,T);H^s(\mathbb{T}^{d};\mathbb{C}^{2}))$. 
We shall prove that $U^N$ converges strongly to $U$.
We fix $\sigma+2+\e\leq s$, $\sigma\geq 2(d+1)+7$, $\e>0$ and  write 
$W:=U-U^N$, then $W$ solves the following problem
\begin{equation*}
\begin{aligned}
\partial_t W&=\ii E \opbw(A(U;x,\xi))W+R(U)W\\
&+\ii E\opbw(A(U)-A(U^N))U^N+(R(U)-R(U^N))U^N\,,
\end{aligned}\end{equation*}
and $W(0,x)=(V_0-V^N_0)(x)$.
We first study the $\s$ norm of the solution $W$. 
If one considers only the first line of the 
equation above then by Lemma \ref{esistenza-gen} 
and Remark \ref{stima-migliore} we have the 
existence of a flow $\phi(t)$ such that
\begin{equation*}
\| \phi(t) W(0,x)\|_{H^{\sigma}}\leq \mathtt{C}_1
(1+\mathtt{C}T)e^{\mathtt{C}T}\|V_0-V_0^N\|_{H^{\s}}\,.
\end{equation*}
By using the Duhamel formulation of the problem  
and  the Lipschitz estimates we obtain 
\begin{equation}\label{bassabassa}
\|W(t)\|_{H^{\sigma}}\leq C_1\|V_0-V_0^N\|_{H^{\s}}
+C_1\int_0^t \big[\|W\|_{H^{\s}}\|U^N\|_{H^{\s+2}}(\tau)
+\|W\|_{H^{\s}}\|U^N\|_{H^{\s}}(\tau)\big] d\tau\,,
\end{equation}
where $C_1>0$ depends on $\|U\|_{H^{\s}}$ and $\|U^{N}\|_{H^{\s}}$ and it is bounded as $U$ goes to $0$.
Note that, since $\s+2< s$, the sequence 
$U^{N}$ is uniformly bounded in $H^{\s+2}(\mathbb{T}^{d};\mathbb{C}^{2})$. 
By Gr\"onwall Lemma we deduce that 
$\|W(t)\|_{H^{\s}}\leq \mathtt{C}_3\|V_0-V_0^N\|_{H^{\s}}$ 
for $\mathtt{C}_3>0$. 
Reasoning analogously for the 
$H^{s}(\mathbb{T}^{d};\mathbb{C}^{2})$ norm one obtains
\begin{equation}\label{sanpellegrino}
\|W(t)\|_{H^s}\leq C\|V_0-V_0^N\|_{H^s}
+C\int_0^t \big[\|W\|_{H^{\s}}\|U^N\|_{H^{s+2}}(\tau)
+\|W\|_{H^s}\|U^N\|_{H^s}(\tau)\big] d\tau\,,
\end{equation}
where $C>0$ depends on $\|U\|_{H^{s}}$ and $\|U^{N}\|_{H^{s}}$ and is bounded as $U$ goes to $0$.
The only unbounded term in the r.h.s. of the latter inequality is 
$\|U^N\|_{H^{s+2}}$. 
To analyze this term one can argue as follows. 
First of all, thanks to \eqref{hystoria}, we have that
$\|U^N\|_{H^{s+2}}\leq \mathtt{C}_4\|V_0^N\|_{H^{s+2}}$, 
where $\mathtt{C}_4$ depends only on $\|V_0^{N}\|_{H^{s}}$.
At this point one wants to use the well known 
smoothing estimate 
$\|V_0^N\|_{H^{s+2}}\lesssim N^2 \|V_0\|_{H^{s}}$. 
To control the loss $N^2$ we use the previous 
estimate we have made on the factor 
$\|W\|_{H^{\s}}\lesssim C\|V_0-V_0^N\|_{H^{\s}}$, 
which may be bounded from above by 
$N^{-2-\varepsilon}\|V_0\|_{H^{s}}$. 
By \eqref{sanpellegrino} we get
\begin{equation}\label{martello2}
\|W(t)\|_{H^s}\leq C\|V_0-V_0^N\|_{H^s}
+C\int_0^t \big[N^{-\e}\|V_0\|^{2}_{H^{s}}
+\|W\|_{H^s}\|U^N\|_{H^s}(\sigma)\big] d\sigma\,.
\end{equation}
Hence we are ready to use Gr\"onwall 
inequality again and conclude the proof.

{\bf Uniqueness.} Let $V_1$ and $V_2$ be two solution of 
\eqref{QLNLS444} with initial condition $V_0$. 
The function $W=V_1-V_2$ solves the problem
\begin{equation*}
\begin{aligned}
\partial_t W&=\ii E \opbw(A(V_1;x,\xi))W+R(V_1)W\\
&+\ii E\opbw(A(V_1)-A(V_2))V_2+(R(V_1)-R(V_2))V_2\,,
\end{aligned}\end{equation*}
with initial condition $W(0,x)=0$.  Arguing as before 
one proves that $\|W(t)\|_{H^{s-2}}=0$ 
for almost every $t$ in $[0,T)$ if $T$ is small enough. 
More precisely one considers the first line 
of the equation and applies Lemma \ref{esistenza-gen} 
and Remark \ref{stima-migliore} 
to obtain a flow of such an equation in 
$H^{s-2}(\mathbb{T}^{d};\mathbb{C}^{2})$ with estimates. 
Then, by means of the  Duhamel formulation of the problem, 
thanks to the fact that the initial condition is equal to zero, 
the estimates on the flow previously obtained 
and Lipschitz estimates, one obtains 
$\|W\|_{H^{s-2}}\leq \frac12\|W\|_{H^{s-2}}$ 
if $T$ is small enough with respect to 
$\|V_1\|_{H^{s}}$ and $\|V_2\|_{H^{s}}$. 
Since $W$ is continuous in time we deduce that is equal to $0$ everywhere. 

{\bf Continuity of the solution map}. 
The strategy is similar to the one adopted in \cite{BSkdv}, \cite{MMT3}.
Let $\{\widetilde{U}_n\}_{n\geq 1}\subset H^s(\T^d;\C^2)$ be a sequence strongly converging to $\widetilde{U}_0$  in $H^s(\T^d;\C^2)$. 
Consider ${U}_n$ and ${U}_0$ 
the  solutions of the problem \eqref{QLNLS444} 
with initial conditions respectively  $\widetilde{U}_n$ and $\widetilde{U}_0$. 
We want to prove that ${U}_n$ converges strongly 
to ${U}_0$ in $H^s(\T^d;\C^2)$. 
Let $T>0$ be small enough and fix  $\varepsilon>0$. 
For $N_{\e}>0$ (to be chosen) we define
\[
\widetilde{U}_{0,\e}:=S_{\leq N_{\e}}\widetilde{U}_0\,,\qquad
\widetilde{U}_{n,\e}:=S_{\leq N_{\e}}\widetilde{U}_n\,,
\]
where $S_{\leq N_{\e}}$ is defined as in \eqref{tronco}
and define ${U}_{n,\e}$ and ${U}_{0,\e}$
the solutions of \eqref{QLNLS444}
with initial conditions 
$\widetilde{U}_{n,\e}$ and $\widetilde{U}_{0,\e}$ respectively.
We note that
\begin{equation}\label{martello}
\|{U}_{n}-{U}_{0}\|_{H^{s}}\leq
\|{U}_{n}-{U}_{n,\e}\|_{H^{s}}+
\|{U}_{n,\e}-{U}_{0,\e}\|_{H^{s}}+
\|{U}_{0,\e}-{U}_{0}\|_{H^{s}}\,.
\end{equation}
Let us consider the first summand in \eqref{martello}.
Let $W:={U}_{n}-{U}_{n,\e}$ then, arguing as done to obtain
the
\eqref{martello2}, one proves
\[
\|W(t)\|_{H^s}\leq C\|\widetilde{U}_{n,\e}-\widetilde{U}_{n}\|_{H^s}
+C\int_0^t \big[N_{\e}^{-\delta}\|\widetilde{U}_{n}\|^{2}_{H^{s}}
+\|W\|_{H^s}\|{U}_{n}\|_{H^s}(\sigma)\big] d\sigma\,,
\]
where  $\delta>0$ is a positive small number.
Since $\widetilde{U}_{n}$ converges to $\widetilde{U}_0$ 
in $H^{s}(\mathbb{T}^{d};\mathbb{C})$
we have, that for $n$ large enough, $\|\widetilde{U}_{n}\|_{H^{s}}
\leq 2\|\widetilde{U}_0\|_{H^{s}}$. 
Then  by \eqref{hystoria} there exists a constant $\Theta$, independent of
$n\in \mathbb{N}$, such that $\|U_{n}\|_{H^{s}}\leq \Theta\|\widetilde{U}_0\|_{H^{s}}$ .
By Gr\"onwall inequality and taking $T>0$ small enough
one gets
\begin{equation}\label{martello4}
\|{U}_{n}-{U}_{n,\e}\|_{H^{s}}\lesssim \|\widetilde{U}_{n,\e}-\widetilde{U}_{n}\|_{H^s}+
N_{\e}^{-\delta}\,.
\end{equation}
Notice that
\[
\begin{aligned}
\|\widetilde{U}_{n,\e}-\widetilde{U}_{n}\|_{H^s}&=\|S_{>N_{\e}}\widetilde{U}_{n}\|_{H^{s}}\\
&\leq \|S_{> N_{\e}}(\widetilde{U}_n-\widetilde{U}_0)\|_{H^{s}}
+\|S_{>N_{\e}}\tilde{U}_{0}\|_{H^{s}}\\
&\leq  \|\widetilde{U}_n-\widetilde{U}_0\|_{H^{s}}+
\|S_{>N_{\e}}\tilde{U}_{0}\|_{H^{s}}\leq \e/6,
\end{aligned}
\]
where to obtain the last inequality we have chosen, independently,
$n$ and $N_{\e}$ large enough.
Then we deduce that the r.h.s. in \eqref{martello4} may be bounded by
$\e/3$ (up to choose a bigger $N_{\e}$).
In the same way one may prove that the third summand in \eqref{martello}
is bounded by $\e/3$ again by choosing $N_{\e}$ large enough.
We now study the second summand in \eqref{martello}.
Arguing as done to obtain the \eqref{sanpellegrino}
we get
\[
\|{U}_{n,\e}-{U}_{0,\e}\|_{H^{s}}\lesssim \|\widetilde{U}_{n,\e}-\widetilde{U}_{0,\e}\|_{H^{s}}
+\int_{0}^{t}
\|{U}_{n,\e}-{U}_{0,\e}\|_{H^{s}}\|U_{0,\e}\|_{H^{s+2}}(\tau)d\tau\,.
\]
By Gr\"onwall inequality and taking $T>0$ small enough
we obtain
\[
\|{U}_{n,\e}-{U}_{0,\e}\|_{H^{s}}\lesssim
\|\widetilde{U}_{n,\e}-\widetilde{U}_{0,\e}\|_{H^{s}}\exp\big(N_{\e}^{2}\big)
\lesssim
\|\widetilde{U}_{n}-\widetilde{U}_{0}\|_{H^{s}}\exp\big(N_{\e}^{2}\big)\,.
\]
By  taking $n\gg N_{\e}$, since $\widetilde{U}_{n}\to\widetilde{U}_0$
in $H^{s}$, we can conclude
\[
\|{U}_{n,\e}-{U}_{0,\e}\|_{H^{s}}\leq \e/3\,.
\]
By \eqref{martello}, we obtain $\|U_n-U_0\|_{H^s}\leq \e$, this implies the thesis.
\end{proof}


\def\cprime{$'$}



\end{document}